\documentclass[twoside, 12pt]{amsart}

\usepackage[utf8]{inputenc}
\usepackage{blindtext}
\usepackage[T1]{fontenc}
\usepackage{amsmath}

\usepackage{amssymb}
\usepackage{MnSymbol}
\usepackage{amsthm}
\usepackage{graphicx}
\usepackage{color,epsfig}      
\usepackage{tikz}
\usetikzlibrary{arrows,shapes}

\newtheorem{Lemma}{Lemma}
\newtheorem{Claim}{Claim}

\newtheorem{Theorem}{Theorem}

\newtheorem{Definition}{Definition}
\newtheorem{Problem}{Problem}

\newtheorem{Corollary}{Corollary}
\newtheorem{Remark}{Remark}

\numberwithin{Subcase}{Case}

\renewcommand{\Re}{\mathbb R}

\newcommand{\K}{\mathbf{K}}

\renewcommand{\S}{\mathbb{S}}

\newcommand{\KK}{\mathcal{K}}

\DeclareMathOperator{\conv}{conv}
\DeclareMathOperator{\perim}{perim}

\DeclareMathOperator{\area}{area}
\DeclareMathOperator{\bd}{bd}

\DeclareMathOperator{\inter}{int}

\DeclareMathOperator{\arclength}{arclength}
\DeclareMathOperator{\cl}{cl}

\begin{document}

\title[Dowker-type theorems]{Dowker-type theorems for disk-polygons in normed planes}

\author[B. Basit]{Bushra Basit}
\author[Z. L\'angi]{Zsolt L\'angi}

\address{Bushra Basit, Department of Algebra and Geometry, Budapest University of Technology and Economics,\\
M\H uegyetem rkp. 3., H-1111 Budapest, Hungary}
\email{bushrabasit18@gmail.com}
\address{Zsolt L\'angi, Department of Algebra and Geometry, Budapest University of Technology and Economics, and MTA-BME Morphodynamics Research Group,\\
M\H uegyetem rkp. 3., H-1111 Budapest, Hungary}
\email{zlangi@math.bme.hu}

\thanks{Partially supported by the National Research, Development and Innovation Office, NKFI, K-147544 grant.}


\begin{abstract}
A classical result of Dowker (Bull. Amer. Math. Soc. 50: 120-122, 1944) states that for any plane convex body $K$ in the Euclidean plane, the areas of the maximum (resp. minimum) area convex $n$-gons inscribed (resp. circumscribed) in $K$ is a concave (resp. convex) sequence. It is known that this theorem remains true if we replace area by perimeter, the Euclidean plane by an arbitrary normed plane, or convex $n$-gons by disk-$n$-gons, obtained as the intersection of $n$ closed Euclidean unit disks. The aim of our paper is to investigate these problems for $C$-$n$-gons, defined as intersections of $n$ translates of the unit disk $C$ of a normed plane. In particular, we show that Dowker's theorem remains true for the areas and the perimeters of circumscribed $C$-$n$-gons, and the perimeters of inscribed $C$-$n$-gons. We also show that in the family of origin-symmetric plane convex bodies, for a typical element $C$ with respect to Hausdorff distance, Dowker's theorem for the areas of inscribed $C$-$n$-gons fails.
\end{abstract}

\maketitle

\section{Introduction}\label{sec:intro}

For any integer $n \geq 3$ and plane convex body $K$, let $A_n(K)$ (resp. $a_n(K)$) denote the the infimum (resp. supremum) of the areas of the convex $n$-gons circumscribed about (resp. inscribed in) $K$. Verifying a conjecture of Kerschner, Dowker \cite{Dowker} proved that for any plane convex body $K$, the sequences $\{ A_n(K) \}$ and $\{ a_n(K) \}$ are convex and concave, respectively. It was proved independently by L. Fejes T\'oth \cite{LFTperim}, Moln\'ar \cite{Molnar} and Eggleston \cite{Eggleston} that the same statements remain true if we replace area by perimeter, where the last author also showed that these statements are false if we replace area by Hausdorff distance. These results are known to be true also in any normed plane \cite{MSW}. Dowker's theorems have became important in many areas of discrete geometry, in particular in the theory of packing and covering \cite{LFTSzeged, Bambah, regfig} and are often used even today (see e.g. \cite{Prosanov, BL23}).

Among many variants of Dowker's theorems that have appeared in the literature, we mention only one, which is related to the notion of spindle convexity. This concept goes back to a paper of Mayer \cite{Mayer} who, for any given convex body $C$ in Euclidean space, considered sets $X$ with the property that for any points $p,q \in X$, $X$ contains the intersection of all translates of $C$ containing $p,q$. He called these sets \emph{hyperconvex}. His paper led to several papers in this topic in the 1930s and 40s, which, however, seems to have been forgotten by the end of the century. In modern times, a systematic investigation of hyperconvex sets was started in the paper \cite{BLNP} in 2007 for the special case that $C$ is a closed Euclidean ball, and a similar paper \cite{LNT2013} appeared in 2013, dealing with any convex body $C$ (see also \cite{JMR}). Some aspects of hyperconvex sets, namely those related to the Carath\'eodory and the Krein-Milman Theorem, are studied in \cite{Pol} in the case that $C$ is a Euclidean ball, and in \cite{BalPol} in general.
Hyperconvex sets have appeared in the literature under several different names: spindle convex, strongly convex or superconvex sets (see e.g. \cite{MM22, Vincensini}), and appear in different areas of mathematics \cite{ChDT, HSTV, MM22}. In this paper, we follow the terminology in \cite{BLNP, LNT2013}, and call a set satisfying the property in Mayer's paper \emph{$C$-spindle convex}, or shortly \emph{$C$-convex}, and if $C$ is a closed Euclidean unit ball, we call it spindle convex (see Definition~\ref{defn:C-spindle}).

One of the results related to spindle convex sets is due to G. Fejes T\'oth and Fodor \cite{TF2015} who extended Dowker's theorems, together with their variants for perimeter, for spindle convex sets; in these theorems the role of inscribed or circumscribed convex $n$-gons is played by the so-called \emph{disk-$n$-gons}, obtained as the intersections of $n$ closed Euclidean unit disks. They also proved similar theorems in hyperbolic and spherical plane, apart from the case of the perimeter of circumscribed polygons on the sphere, which is still open.

Our main goal is to investigate a normed version of the problem in \cite{TF2015}, by replacing the Euclidean unit disks by the unit disks of a normed plane, and measuring area and perimeter according to the norm. To state our results, recall that the unit ball of any finite dimensional normed space is a convex body symmetric to the origin $o$, and any such body is the unit ball of a finite dimensional normed space. Thus, in the paper we choose an arbitrary $o$-symmetric convex disk $C$ in the real normed space $\Re^2$, and work in the normed plane with unit disk $C$, which we regard as $\Re^2$ equipped with the norm $||\cdot||_C$ of $C$.
In the paper, by a convex disk we mean a compact, convex planar set with nonempty interior. We denote the family of convex disks by $\KK$, and the family of $o$-symmetric convex disks by $\KK_o$. In the paper we regard $\KK$ and $\KK_o$ as topological spaces with the topology induced by Hausdorff distance.

Before presenting our results, recall the well-known fact that any finite dimensional real normed space can be equipped with a Haar measure, and that this measure is unique up to multiplication of the standard Lebesgue measure by a scalar (cf. e.g. \cite{Thompson}). This scalar does not play a role in our investigation and in the paper $\area(\cdot)$ denotes $2$-dimensional Lebesgue measure. 

\begin{Definition}\label{defn:Cperim}
For any $C \in \KK_o$ and convex polygon $Q$, we define the $C$-perimeter of $Q$ as the sum of the lengths of the sides of $Q$, measured in the norm generated by $C$. The  \emph{$C$-perimeter} of a convex disk $K \subset \Re^2$, denoted by $\perim_C(K)$, is the supremum of the $C$-perimeters of all convex polygons inscribed in $K$.
\end{Definition}

We note that, moving its vertices one by one to the boundary of $K$ in a suitable direction, for any convex polygon $Q$ contained in $K$ one can find a convex polygon $Q'$ inscribed in $K$ with $\perim_C(Q) \leq \perim_C(Q')$. This shows, in particular, that for any two plane convex bodies $K \subseteq L \subset \Re^2$, we have $\perim_C(K) \leq \perim_C(L)$, with equality if and only if $K=L$ (see \cite{MSW}, and for this argument, also \cite[Theorem 5B]{Schaffer}).
Furthermore, it is worth observing that a straightforward modification of Definition~\ref{defn:Cperim} can be used to define the \emph{$C$-length} of a rectifiable curve $\Gamma \subset \Re^2$, denoted by $\arclength_C(\Gamma)$.

Our next definition can be found in \cite{LNT2013} and its origin goes back to \cite{Mayer}.

\begin{Definition}\label{defn:C-spindle}
Let $C \in \KK_o$ and consider two (not necessarily distinct) points $p, q \in \Re^2$ such that a translate of $C$ contains both $p$ and $q$.
Then the \emph{$C$-spindle} (denoted as $[p,q]_C$) of $p$ and $q$ is the intersection of all translates of $C$
that contain $p$ and $q$. If no translate of $C$ contains $p$ and $q$, we set $[p,q]_C = \Re^2$.
We call a set $K \subset \Re^2$ \emph{$C$-spindle convex} (or shortly \emph{$C$-convex}), if for any $p,q \in K$, we have $[p,q]_C \subseteq K$.
\end{Definition}

We recall from \cite[Corollary 3.13]{LNT2013} that a closed set in $\Re^2$ different from $\Re^2$ is $C$-convex if and only if it is the intersection of some translates of $C$.

\begin{Definition}\label{defn:Cpolygon}
The intersection of $n$ translates of $C$ is called a \emph{$C$-$n$-gon} for $n \geq 3$.
\end{Definition}

In our next definition and throughout the paper, $\area(\cdot)$ denotes standard Lebesgue measure.

\begin{Definition}\label{defn:areper}
Let $n \geq 3$ and let $K$ be a $C$-convex disk in $\Re^2$, where $C \in \KK_o$. We set
\begin{equation}\label{eq:Cconv_infsup}
\begin{aligned}
\hat{A}_n^C(K) = & \inf \{ \area(Q) : Q \hbox{ is a } C-n-\hbox{gon circumscribed about } K \};\\
\hat{a}_n^C(K) = & \sup \{ \area(Q) : Q \hbox{ is a } C-n-\hbox{gon inscribed in } K \};\\
\hat{P}_n^C(K) = & \inf \{ \perim_C(Q) : Q \hbox{ is a } C-n-\hbox{gon circumscribed about } K \};\\
\hat{p}_n^C(K) = & \sup \{ \perim_C(Q) : Q \hbox{ is a } C-n-\hbox{gon inscribed in } K \}.
\end{aligned}
\end{equation}
\end{Definition}

The next theorem, together with Theorem~\ref{thm:counter1}, is the main result of the paper. We note that other results related to the approximation of $C$-convex disks by $C$-$n$-gons can be found also in \cite{Fodor, NV}.

\begin{Theorem}\label{thm:Dowker1}
For any $C \in \KK_o$ and $C$-convex disk $K$, the sequences $\{ \hat{A}_n^C(K) \}$, $\{ \hat{P}_n^C(K) \}$ are convex, and the sequence $\{ \hat{p}_n^C(K) \}$ is concave. That is, for any $n \geq 4$, we have
\[
\hat{A}_{n-1}^C(K)+\hat{A}_{n+1}^C(K) \geq 2 \hat{A}_n^C(K), \hat{P}_{n-1}^C(K)+\hat{P}_{n+1}^C(K) \geq 2 \hat{P}_n^C(K), \hbox{ and}
\]
\[
\hat{p}_{n-1}^C(K)+\hat{p}_{n+1}^C(K) \leq 2 \hat{p}_n^C(K).
\]
\end{Theorem}

As a consequence of Theorem~\ref{thm:Dowker1}, we prove Theorem~\ref{thm:Dowker2}, and recall that similar statements have been derived in \cite{GFTandLFT} for the Euclidean areas of inscribed and circumscribed polygons from the classical results of Dowker in \cite{Dowker} (for their spindle convex variants, see \cite{TF2015}).
Before stating it, we recall that a plane convex body $K$ has a $k$-fold rotational symmetry with respect to the point $p$ if the rotation around $p$ with angle $\frac{2\pi}{k}$ in counterclockwise direction is a symmetry of $K$.

\begin{Theorem}\label{thm:Dowker2}
Let $n \geq 3$ and $k \geq 2$. Assume that $k$ is a divisor of $n$ and both $K$ and $C$ have $k$-fold rotational symmetry. Then there are $C$-$n$-gons $Q^A$, $Q^P$ circumscribed about $K$ which have $k$-fold rotational symmetry, and $\area(Q^A)= \hat{A}_n^C(K)$ and $\perim_C(Q^P)= \hat{P}_n^C(K)$. Similarly, there is a $C$-$n$-gon $Q^p$ inscribed in $K$ which has $k$-fold rotational symmetry, and $\perim_C(Q^p)= \hat{p}_n^C(K)$.
\end{Theorem}

Before our next theorem, we remark that in a topological space $\mathcal{F}$, a subset is called \emph{residual} if it is a countable intersection of sets each of which has dense interior in $\mathcal{F}$. The elements of a residual subset of $\mathcal{F}$ are called \emph{typical}.
Based on the results in \cite{Dowker, TF2015, MSW} mentioned in the previous part of the paper, one would expect that for any $C \in \KK_o$ and $C$-convex disk $K$, the sequence $\{ \hat{a}_n^C(K) \}$ is concave. Our next result shows that this property fails for the sequence $\{ a_n^C(K) \}$ for a \emph{typical} element of $\KK_o$.

\begin{Theorem}\label{thm:counter1}
A typical element $C$ of $\KK_o$ satisfies the property that for every $n \geq 4$, there is a $C$-convex disk $K$ with
\[
\hat{a}_{n-1}^C(K) + \hat{a}_{n+1}^C(K) > 2 \hat{a}_n^C(K).
\]
\end{Theorem}

\begin{Remark}\label{rem:normed1}
For $C \in \KK_o$, $K \in \KK$ and positive integer $n \geq 3$, let
\begin{equation}\label{eq:Cper_cr}
\bar{P}_n^C(K) = \inf \{ \perim_C(Q) : Q \hbox{ is a convex } n- \hbox{gon circumscribed about } K \};
\end{equation}
\begin{equation}\label{eq:Cper_ir}
\bar{p}_n^C(K) = \sup \{ \perim_C(Q) : Q \hbox{ is a convex } n- \hbox{gon inscribed in } K \}.
\end{equation}
As we have observed, it is known \cite{MSW} that for any $C \in \KK_o$ and $K \in \KK$, the sequences $\{ \bar{P}_n^C(K) \}$ and $\{ \bar{p}_n^C(K) \}$ are convex and concave, respectively. Our approach yields a new proof of these statements by applying Theorem~\ref{thm:Dowker1} for $\lambda C$, where $\lambda \to \infty$.
\end{Remark}

Applying Theorem~\ref{thm:Dowker2} for $\lambda C$ with $\lambda \to \infty$, we obtain the following.

\begin{Remark}\label{rem:normed2}
Let $C \in \KK_o$, $K \in \KK$ and $n \geq 3$. If, for some positive integer $k$, 
Let $C \in \KK_o$, $K \in \KK$, $n \geq 3$ and $k \geq 2$. Assume that $k$ is a divisor of $n$ and both $K$ and $C$ have $k$-fold rotational symmetry. Then there is a convex $n$-gon $Q^P$ circumscribed about $K$ with $\perim_C(Q^P)= \bar{P}_n^C(K)$ such that $Q^P$ has $k$-fold rotational symmetry. Similarly, there is a convex $n$-gon $Q^p$ inscribed in $K$ which has $k$-fold rotational symmetry, and $\perim_C(Q^p)= \bar{p}_n^C(K)$.
\end{Remark}

The structure of the paper is as follows. In Section~\ref{sec:prelim}, we present the necessary notation and prove some lemmas. Then in Sections~\ref{sec:proofs} and \ref{sec:counter1} we prove Theorems~\ref{thm:Dowker1} and \ref{thm:Dowker2}, and Theorem~\ref{thm:counter1}, respectively. Finally, in Section~\ref{sec:remarks}, we 
propose some open problems.

\section{Preliminaries}\label{sec:prelim}

In Subsection~\ref{subsec:prelim1} we collect the necessary properties of $C$-convex disks in general. In Subsection~\ref{subsec:prelim2} we investigate the properties of the quantities in Definition~\ref{defn:areper}, and the lemmas closely related to the proofs of Theorems~\ref{thm:Dowker1}-\ref{thm:counter1}.

\subsection{General properties of $C$-convex disks}\label{subsec:prelim1}

In the paper, for simplicity, for any $x,y \in \Re^2$, we denote by $[x,y]$ the closed segment with endpoints $x,y$. We equip $\Re^2$ also with a Euclidean norm, which we denote by $||\cdot||$, and use the notation $B^2$ for the Euclidean closed unit disk centered at $o$. Recall that the \emph{Euclidean diameter} of a compact set $X \subset \Re^2$ is the Euclidean distance of a farthest pair of points in $X$. If we replace Euclidean distance by distance measured in the norm of $C$, we obtain the \emph{$C$-diameter} of $X$.

Recall that for any set $X \subseteq \Re^2$, the \emph{$C$-convex hull}, or shortly \emph{$C$-hull} is the intersection of all $C$-convex sets that contain $X$. We denote it by $\conv_C(X)$, and note that it is $C$-convex, and if $X$ is compact, then it coincides with the intersection of all translates of $C$ containing $X$ \cite{LNT2013}.

\begin{Remark}\label{rem:summand}
It is worth noting that the property whether a convex disk is $C$-convex or not can also be phrased in terms of Minkowski sums. Indeed, following \cite{Schneiderbook}, we call a compact, convex set $K$ a \emph{summand} of a compact convex set $L$ if there is a compact, convex set $M$ with $K+M = L$. Then, by \cite[Theorem 3.2.2]{Schneiderbook}, a convex disk $K$ is $C$-convex if and only if $K$ is a summand of $C$.
\end{Remark}

In the following list we collect some elementary properties of $C$-spindles and $C$-$n$-gons that we are going to use frequently in the paper. 

\begin{Lemma}\label{lem:Cgons}
We have the following.
\begin{itemize}
\item[(a)] For any $x, y \in \Re^2$, $[x,y]_C=[x,y]$ if and only if a translate of $C$ contains $[x,y]$ in its boundary.
\item[(b)] If $[x,y]_C \neq [x,y]$, then $[x,y]_C$ is a centrally symmetric convex disk whose boundary consists of two arcs, connecting $x$ and $y$, that are contained in the boundary of some translates of $C$.
\item[(c)] For any $x,y \in \Re^2$ with $||x-y||_C \leq 2$, $[x,y]_C$ is the intersection of at most two translates of $C$, and if $[x,y]_C$ is a translate of $C$, then $||x-y||_C=2$.
\item[(d)] Conversely, a nonempty intersection of at most two translates of $C$ is the $C$-spindle of two (not necessarily distinct) points.
\item[(e)] A set is a $C$-$n$-gon if and only if it is the $C$-hull of at most $n$ points contained in a translate of $C$.
\end{itemize}
\end{Lemma}

\begin{proof}
Clearly, to prove (a)-(c), we can assume that $x \neq y$. Let $L_1$ and $L_2$ be the two supporting lines of $C$ parallel to $[x,y]$. Then the segments $L_1 \cap C$ and $L_2 \cap C$ have the same $C$-length $l \geq 0$. If $l \geq ||x-y||_C$, then $[x,y]_C = [x,y]$. Consider the case that $l < ||x-y||_C < 2$. Then there are exactly two chords $X_1, X_2$ of $C$ which are translates of $[x,y]$. For $\{ i,j \} = \{ 1,2 \}$, the chord $X_i$ dissects $C$ into two convex disks exactly one of which contains $X_j$. Let us denote the part not containing $X_j$ by $C_i$. Note that for any $x \in \Re^2$, the property that $x+X_i \subset C$ implies that $x+C_i \subset C$. Thus, in this case $[x,y]_C$ is the union of a translate of $C_1$ and a translate of $C_2$. This shows that $[x,y]_C$ is not a segment, and its boundary consists of the translates of two arcs of $\bd (C)$. If $||x-y||_C=2$, we may repeat the same consideration, where as $X_1$ and $X_2$ we pick, from amongst the chords of $C$ that are translates of $[x,y]$, the two that are farthest from each other. These arguments prove (a)-(c).

To prove (d), consider the set $K=C_1 \cap C_2 \neq \emptyset$, where $C_1 = p+C$ and $C_2=q+C$ for some $p \neq q$. If $K$ is a point or a segment, the statement is clearly satisfied. Otherwise, $K$ is a convex disk symmetric to some point $c$, and $\bd (K) = \left( K \cap \bd(C_1) \right) \cup \left( K \cap \bd(C_2) \right)$. Note that $K \cap \bd(C_1)$ is a closed, continuous curve, and $K \cap \bd(C_2)$ is its reflected copy about $c$. Furthermore, $K \cap \bd(C_1) \cap \bd(C_2)$ is a pair of points or a pair of parallel segments. If it is a pair of points $\{ x,y \}$, then, by the arguments in the previous paragraph, $K=[x,y]_C$. If it is a pair of segments $[x,y]$ and $[w,z]$, then $[w,z]$ is a translate of $[x,y]$. Assuming, without loss of generality, that $x-y=w-z$, we obtain that $K=[x,z]_C$. This proves (d).

To prove (e), we can apply a slight modification of the arguments for (a)-(d). 
\end{proof}

\begin{Lemma}\label{lem:continuity}
Let $x,y \in C \in \KK_o$, with $||x-y||_C < 2$. Then, for any sequences $x_m \to x$, $y_m \to y$, $C_m \to C$ with $x_m,y_m \in \Re^2$ and $C_m \in \KK_o$, where the limit for the $C_m$ is meant with respect to Hausdorff distance, we have $[x_m,y_m]_{C_m} \to [x,y]_C$. 
\end{Lemma}

\begin{proof}
Since $||x-y||_C < 2$, without loss of generality we may assume that $||x_m-y_m||_{C_m} < 2$ for all values of $m$. Consider the case that $[x,y]_C \neq [x,y]$. Then
we may assume that $[x_m,y_m]_{C_m} \neq [x_m,y_m]$ for all values of $m$. Let $X_m^1$ and $X_m^2$ denote the chords of $C_m$ that are translates of $[x_m,y_m]$, and define $X^1$ and $X_2$ similarly for $C$. Note that there is a line $L$ through $o$ that separates $X^1$ and $X^2$, and also $X_m^1$ and $X_m^2$ for all sufficiently large values of $m$. Thus, we may assume that $X^1$ and all $X_m^1$ are on the same side of $L$, and they are separated from $X^2$ and all $X_m^2$. Then, as $m \to \infty$, $X_m^1 \to X_1$ and $X_m^2 \to X_1$, which yields the statement in this case.

Assume that $[x,y]_C = [x,y]$ and $x \neq y$. Then a similar consideration can be applied, based on the observation that any sequence $\{ X_m \}$, where $X_m$ is a chord of $C_m$ that is a translate of $[x_m,y_m]$, the Euclidean distance of $X_m$ from the closer supporting line of $C_m$ parallel to $[x_m,y_m]$ tends to zero.
Finally, if $x=y$, then the statement is trivial.
\end{proof}

We observe that the statement in Lemma~\ref{lem:continuity} does not necessarily hold if $||x-y||_C = 2$. As an example, we can choose $C$ as a parallelogram, $x_m=x$ and $y_m=y$ as the midpoints of two opposite sides $S_1, S_2$ of $C$, and $\{ C_m \}$ as a sequence of $o$-symmetric hexagons inscribed in $C$ whose elements intersect $S_1$ and $S_2$ only in $x$ and $y$, respectively.

\begin{Lemma}\label{lem:smoothisdense}
Assume that $C$ has $C^{\infty}$-class boundary and strictly positive curvature. Then the family of $C$-convex disks with $C^\infty$-class boundary is an everywhere dense set in the family of $C$-convex disks.
\end{Lemma}

\begin{proof}
First, observe for any $C$-convex disk $K$, $\lambda K$ is $C$-convex for any $0 < \lambda < 1$. Furthermore, $C$-$n$-gons form an everywhere dense set in the family of $C$-convex disks. Thus, to prove the statement, it is sufficient to show that for every $0 < \lambda < 1$, every $(\lambda C)$-$n$-gon can be approximated arbitrarily well by a $C$-convex disk with $C^{\infty}$-class boundary. Let $0 < \lambda < 1$ be fixed, and let $P$ be an arbitrary $(\lambda C)$-$n$-gon. Let $0 < \varepsilon$. Note that apart from its vertices, $P$ has $C^{\infty}$-class boundary. Now we `smoothen' the vertices of $P$ by replacing a neighborhood of every vertex $p_i$ of $P$ in in $\bd (P)$ by a translate $T_i$ of an arc of $\tau_i \bd (C)$, where $\tau_i > 0$ is sufficiently small. Clearly, we can do it in such a way that the resulting curve is the boundary of a convex disk $K'$ with $C^1$-class boundary. Note that apart from the points where a side of $P$ meets one of the $T_i$, $\bd (K')$ has strictly greater curvature than the point of $\bd (C)$ with the same outer unit normal, and the same holds for the one-sided curvatures of $\bd (K')$ at the remaining points. Furthermore, we may assume that the Hausdorff distance of $K'$ and $P$ less than $\frac{\varepsilon}{2}$.

Now, we make $\bd (K')$ $C^{\infty}$-class using a standard convolution technique, described e.g. in \cite{Ghomi}.
Assume that $o \in \inter (K')$, and let $\rho_{K'} : \Re \to \Re$ be the function defined as $\rho_{K'}(\varphi) = \max \{ \lambda : \lambda (\cos \varphi, \sin \varphi) \in K' \}$. Consider a nonnegative smooth function $\delta : \Re \to \Re$ such that $\int_{-\infty}^{\infty} \delta(t) \, dt = 1$, and its support is contained in $(1,1)$. For any $\tau > 0$, let $\delta_{\tau}(t) = \delta \left( \frac{t}{\tau} \right)$, and note that the support of $\delta_{\tau}$ lies in $(-\tau,\tau)$.

Set $\rho^{\tau}_{K'}(\varphi) = \int_{-\infty}^{\infty} \rho_{K'}(\varphi-t) \delta_{\tau}(t) \, dt$.
(We can geometrically imagine this transformation as taking an integral average of the values of $\rho_{\K'}$ in the $\tau$-neighborhood of $\varphi$.)
By the basic properties of convolution,
\[
\partial^{(k)} \int_{-\infty}^{\infty} \rho_{K'}(\varphi-t) \delta_{\tau}(t) \, dt = \partial^{(k)} \int_{-\infty}^{\infty} \rho_{K'}(s) \delta_{\tau}(\varphi-s) \, ds
=  \int_{-\infty}^{\infty} \rho_{K'}(s) \partial^{(k)} \delta_{\tau}(\varphi-s) \, ds,
\]
and thus, $\rho^{\tau}_{K'}$ is $C^{\infty}$-class, and as $\tau \to 0$, $\rho_{K'}^{\tau} \to \rho_{K'}$ and $\left( \rho_{K'}^{\tau} \right)' \to \rho_{K'}'$ uniformly, and if $\tau > 0$ is sufficiently small, then for any value of $\varphi$, the curvature of $\rho^{\tau}_{K'}$ is strictly greater than the curvature of $\bd (C)$ at the point with the same outer normal. In particular, it follows that the plane curve $\rho^{\tau}_{K'} (\varphi) (\cos \varphi, \sin \varphi)$ is the boundary of a convex disk $K$, and by \cite[Theorem 3.2.12]{Schneiderbook}, $K$ is a summand of $C$, and thus, it is $C$-convex (see Remark~\ref{rem:summand}).  
\end{proof}

\begin{Definition}\label{defn:arclengthC}
Let $C \in \KK_o$, and let $x,y$ be points with $||x-y||_C \leq 2$.
Then the \emph{arc-distance $\rho_C (x,y)$ of $x,y$ with respect to $C$} (or shortly, \emph{$C$-arc-distance of $x$ and $y$}) is the minimum of the $C$-lengths of the arcs, with endpoints $x,y$, that are contained in $z+\bd(C)$ for some $z \in \Re^2$.
\end{Definition}

\begin{Remark}\label{rem:arclengthC}
For any $x,y \in \Re^2$ with $||x-y||_C \leq 2$, if $[x,y]_C \neq [x,y]$, then $\rho_C (x,y) = \frac{1}{2} \perim_C ([p,q]_C)$. Furthermore, if $[x,y]_C = [x,y]$, then $\rho_C(x,y)=||x-y||_C$.
\end{Remark}

We recall the following version of the triangle inequality from \cite[Theorem 6]{LNT2013}.

\begin{Lemma}[L\'angi, Nasz\'odi, Talata]\label{lem:triangle_inequality}
Let $C \in \KK_0$, and let $x,y,z$ be points such that each pair has a $C$-arc-distance.
\begin{itemize}
\item[(a)] If $y \in \inter [x,z]_C$, then $\rho_C(x,y)+\rho_C(y,z) \leq \rho_C(x,z)$.
\item[(b)] If $y \in \bd [x,z]_C$, then $\rho_C(x,y)+\rho_C(y,z) = \rho_C(x,z)$.
\item[(c)] If $y \notin [x,z]_C$ and $C$ is smooth, then $\rho_C(x,y)+\rho_C(y,z) \geq \rho_C(x,z)$.
\end{itemize}
\end{Lemma}

We start with a consequence of this inequality. We note that for the special case that $C$ is a Euclidean disk, Lemma~\ref{lem:quadrangle} was proved in \cite{BCC2006}.

\begin{Lemma}\label{lem:quadrangle}
Let $p,q,r,s \in \Re^2$ be distinct points contained in a translate of the smooth $o$-symmetric convex disk $C$. Assume that $\bd \conv_C \{p,q,r,s\}$ contains all of them and that $p$ and $r$ separate $q$ and $s$ in $\bd \conv_C \{p,q,r,s\}$. Then
\[
\rho_{C}(p,q)+\rho_{C}(r,s) \leq \rho_{C}(p,r)+\rho_{C}(q,s).
\]
\end{Lemma}

\begin{proof}
We prove the statement under the special condition that none of $p,q,r,s$ is contained in the $C$-convex hull of the other three points, $[p,r]_C \neq [p,r]$, and $[q,s]_C \neq [q,s]$, as in the opposite cases a similar argument can be applied. Then,
according to our conditions, the two $C$-arcs in the boundary of $[p,r]_C$ intersect both $C$-arcs consisting of the boundary of $[q,s]_C$. Let $s'$ denote the intersection point of one of the $C$-arcs in $\bd [p,r]_C$ and one of the $C$-arcs in $\bd [q,s]_C$, where the arcs are chosen to satisfy $s' \in \bd \conv_C \{ p,q,s \}$ and $s' \in \bd  \conv_C \{p,r,s\}$ (see Figure~\ref{fig:lemma5}). (Here we mention that since any two $C$-arcs intersect in at most two, possibly degenerate, segments, $s'$ uniquely exists.) Then, by the definition of $C$-convex hull,
$s' \notin \inter [p,q]_C$ and $s' \notin \inter [r,s]_C$. Since $[q,s']_C \subset [q,s]_C$ and $[p,s']_C \subset [p,r]_C$, we also have $p \notin [q,s']_C$ and $q \notin [p,s']_C$ and $p \notin [q,s']_C$, which yields that $p,q,s'$ are in $C$-convex position. Applying a similar argument, we obtain that also $r,s,s'$ are in $C$-convex position.
Thus, by Lemma~\ref{lem:triangle_inequality}, we have $\rho_C(p,q) \leq \rho_C(p,s')+\rho_C(q,s')$ and $\rho_C(r,s) \leq \rho_C(r,s') + \rho_C(s,s')$, implying the assertion.
\end{proof}

\begin{figure}[ht]
\begin{center}
 \includegraphics[width=0.3\textwidth]{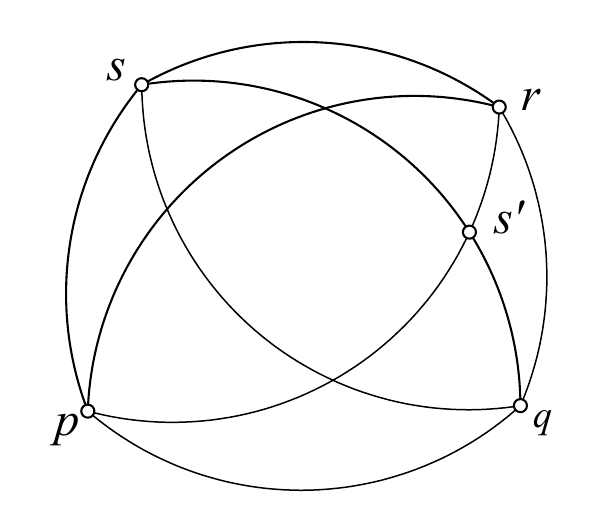}
 \caption{An illustration for the proof of Lemma~\ref{lem:quadrangle}, with a Euclidean disk as $C$.}
\label{fig:lemma5}
\end{center}
\end{figure}

\subsection{Necessary lemmas for the proofs of Theorems~\ref{thm:Dowker1}-\ref{thm:counter1}}\label{subsec:prelim2}

For any $n \geq 4$, let $\KK^n_a$ denote the subfamily of the elements $C$ of $\KK_0$ satisfying the Dowker-type inequality $\hat{a}_{n-1}^{C}(K) + \hat{a}_{n+1}^{C}(K) \leq 2 \hat{a}_{n}^{C}(K)$ for any $C$-convex disk $K$. We define $\KK^n_A$, $\KK^n_p$ and $\KK^n_P$ similarly.
Our first lemma describes the topological properties of these families.

\begin{Lemma}\label{lem:closed}
For any $n \geq 4$, $\KK^n_a, \KK^n_A, \KK^n_p$ and $\KK^n_P$ are closed.
\end{Lemma}

\begin{proof}
We prove the assertion only for $\KK^n_a$, as for the other quantities the proof is analogous.
Let $C \notin \KK^n_a$, and suppose for contradiction that there is a sequence $C_m \in \KK_a^n$ with $C_m \to C$.
Since $C \notin \KK^n_a$, there is a $C$-convex disk $K$ satisfying $\hat{a}_{n-1}^{C}(K) + \hat{a}_{n+1}^{C}(K) > 2 \hat{a}_{n}^{C}(K)$. By Lemma~\ref{lem:Cgons}, if $K$ contains points at $C$-distance equal to $2$, then $K$ is a $C$-spindle, which yields that $\hat{a}_j(K) = \area(K)$ for any $j \geq 3$. Thus, according to our assumptions, $K$ does not contain points at $C$-distance equal to $2$, i.e its $C$-diameter is strictly less than $2$. On the other hand, since $K$ is $C$-convex, $K$ is the intersection of the translates of $C$ that contain it. Thus, there is a set $X \subset \Re^2$ such that $K = \bigcap_{x \in X} (x+C)$. 

Let $K_m = \bigcap_{x \in X} (x+C_m)$. Then, clearly, $K_m$ is $C_m$-convex, and $K_m \to K$. For $j=n-1,n+1$, let $Q_j$ be a $C$-$j$-gon inscribed in $K$ such that $\area(Q_j)=\hat{a}_j^{C}(K)$. Then, as $K_m \to K$ and $C_m \to C$, there are sequences $\{ Q_{n-1}^m \}$ and $\{ Q_{n+1}^m \}$ such that for $j=n-1,n+1$, $Q_j^m$ is a $C_m$-$j$-gon inscribed in $K_m$, and $Q_j^m \to Q_j$. By the properties of Hausdorff distance, the $C_m$-diameter of $K_m$ is strictly less than $2$ if $m$ is sufficiently large. Then we can apply Lemma~\ref{lem:continuity}, and obtain that $\area(Q_j^m) \to \area(Q_j)$ for $j=n-1,n+1$. From this, we have $\area(Q_{n-1}^m)+\area(Q_{n+1}^m) \to \hat{a}_{n-1}^{C}(K) + \hat{a}_{n+1}^{C}(K)$. On the other hand, since $C_m \in \KK^n_a$, there is a sequence $\{ Q_n^m \}$ such that $Q_n^m$ is a $C_m$-$n$-gon inscribed in $K_m$, and $2 \area(Q_n^m) \geq \area(Q_{n-1}^m)+\area(Q_{n+1}^m)$. By compactness, we may assume that $\{ Q_n^m \}$ converges to a $C$-$n$-gon $Q_n$. Clearly, $Q_n$ is contained in $K$, and by Lemma~\ref{lem:continuity}, $\area(Q_n^m) \to \area(Q_n)$. Thus, $\hat{a}_{n-1}^{C}(K) + \hat{a}_{n+1}^{C}(K) \leq 2 \area(Q_n) \leq 2 \hat{a}_n^{C}(K)$; a contradiction.
\end{proof}

Lemma~\ref{lem:closed} readily yields Corollary~\ref{cor:closed}, since the intersection of arbitrarily many closed sets is closed.

\begin{Corollary}\label{cor:closed}
The family $\bigcap_{n=4}^{\infty} \KK_a^n$ of the elements $C$ of $\KK_0$ satisfying $\hat{a}_{n-1}^{C}(K) + \hat{a}_{n+1}^{C}(K) \leq 2 \hat{a}_{n}^{C}(K)$ for all $n \geq 4$ and all $C$-convex disks $K$ is closed in $\KK_0$. Similar statements hold for the families $\bigcap_{n=4}^{\infty} \KK_p^n$, $\bigcap_{n=4}^{\infty} \KK_A^n$ and $\bigcap_{n=4}^{\infty} \KK_P^n$.
\end{Corollary}
 
\begin{Lemma}\label{lem:defncontinuity}
If $\bd(C)$ is $C^{\infty}$-class with strictly positive curvature, the quantities defined in Definition~\ref{defn:areper} are continuous functions of $K$ for any fixed value of $n$.
\end{Lemma}

\begin{proof}
We show this only for the quantity $\hat{a}_n^C(K)$, as for the other three quantities a similar argument can be given. Now, consider a convergent sequence $K_m \to K$, where $K_m, K$ are $C$-convex disks for every $m$, and the limit is meant with respect to Hausdorff distance. Let $n \geq 3$ be fixed.
If there are some $x,y \in K$ with $||x-y||_C=2$, then by Lemma~\ref{lem:Cgons}, $K$ is the intersection of at most two translates of $C$, and we may assume that $K=[x,y]_C$. In this case $a_n(K)= \area(K)$ for any $n \geq 3$, and if we choose points $x_m \to x$ and $y_m \to y$ with $x_m, y_m \in K_m$, then $[x_m,y_m] \to [x,y]$, 
and our assumptions for $C$ imply that $[x_m,y_m]_C \to [x,y]_C$. Thus, in this case $\hat{a}_n^C(K_m) \to \hat{a}_n^C(K)$.
Next, assume that $K$ contains no points $x,y$ with $||x-y||_C=2$. By compactness, we can find points $p_i^m$, $i=1,2,\ldots,n$, satisfying $\area (\conv_C \{ p_1^m,p_2^m,\ldots, p_n^m\}) = \hat{a}_n^C(K_m)$. Without loss of generality, we may assume that the sequences $\{ p_i^m \}$ are convergent. Let $p_i^m \to p_i$. By Lemma~\ref{lem:continuity}, we have that $\area (\conv_C \{ p_1^m,p_2^m,\ldots, p_n^m\}) \to \area (\conv_C \{ p_1,p_2,\ldots, p_n\})$, implying that $\hat{a}_n^C(K) \geq \lim_{m \to \infty} \hat{a}_n^C(K_m)$. To prove the inequality $\hat{a}_n^C(K) \leq \lim_{m \to \infty} \hat{a}_n^C(K_m)$, we take a maximal area $C$-$n$-gon $\conv_C \{ q_1,q_2,\ldots, q_n\}$ in $K$, choose sequences $q_i^m \to q_i$ with $q_i^m \in K_m$ and $i=1,2,\ldots, n$, and again apply Lemma~\ref{lem:continuity}.
\end{proof}

In the following lemma, let $\S^1$ denote the Euclidean unit circle centered at the origin. For simplicity, if $x,y \in \S^1$, we denote by $\widehat{xy}$ the Euclidean closed circle arc obtained as the orbit of $x$ when it is rotated around $o$ in counterclockwise direction until it reaches $y$. If for two such arcs $\widehat{xy} \subseteq \widehat{zw} \subset \S^1$, we say that $\widehat{zw}$ \emph{covers} $\widehat{xy}$, and call such a pair of arcs a \emph{covering pair}. Let $\mathcal{S}$ denote the family of closed circle arcs $\widehat{xy}$ of $S$. Furthermore, we say that a function $f : \mathcal{S} \to \Re$ has a $k$-fold rotational symmetry for some positive integer $k$, if for any $S,S' \in \mathcal{S}$, where $S'$ is a rotated copy of $S$ in counterclockwise direction with angle $\frac{2\pi}{k}$, we have $f(S)=f(S')$. Lemma~\ref{lem:functional_1} can be regarded as a functional form of Dowker's theorems, and its proof is based on the idea of the original theorem of Dowker \cite{Dowker}.

\begin{Lemma}\label{lem:functional_1}
Let $f : \mathcal{S} \to \Re$ be a bounded function with $f(\widehat{xx})=0$ for all $x \in \S^1$. For any integer $n \geq 3$, let
\[
M_n = \sup \{ \sum_{S \in X} f( S ) : X \subset \mathcal{S} \hbox{ is a tiling of } \S^1 \hbox{ with } |X| = n \}.
\]
If for any $\widehat{x_2x_3} \subset \widehat{x_1x_4}$, we have
\[
f(\widehat{x_1x_3})+f(\widehat{x_2x_4}) \geq f(\widehat{x_1x_4})+f(\widehat{x_2x_3}),
\]
then the sequence $\{ M_n \}$ is concave.
Furthermore, if in addition, there is some positive integer $k$ such that $k | n$ and $f$ has $k$-fold rotational symmetry, and there is an $n$-element tiling  $X$ of $\S^1$ such that $M_n = \sum_{S \in X} f(S)$ then there is an $n$-element tiling $X'$ of $\S^1$ with $k$-fold rotational symmetry such that $M_n = \sum_{S \in X'} f(S)$.
\end{Lemma}

Before the proof, we remark that $X \subset \mathcal{S}$ is called an $m$-tiling of $\S^1$ for some positive integer $m$ if every point of $\S^1$ belongs to at least $m$ members of $X$, and to the interiors of at most $m$ members of $X$.

\begin{proof}
To prove the assertion for $\{ M_n \}$, we need to show that $M_{n-1}+M_{n+1} \leq 2M_n$ is satisfied for any $n \geq 4$. In other words, we need to show that for any tilings $X=\{ \widehat{x_0x_1}, \ldots \widehat{x_{n-2}x_{n-1}} \}$, $Y=\{ \widehat{y_0y_1}, \ldots \widehat{y_n y_{n+1}} \}$ of $\S^1$, there are tilings $Z=\{ \widehat{z_0z_1}, \ldots \widehat{z_{n-1}z_n} \}$ and $W=\{ \widehat{w_0w_1}, \ldots \widehat{w_{n-1}w_n} \}$ of $\S^1$ such that
\[
\sum_{i=1}^{n-1} f(\widehat{x_{i-1}x_i}) + \sum_{i=1}^{n+1} f(\widehat{y_{i-1}y_i}) \leq \sum_{i=1}^{n} f(\widehat{z_{i-1}z_i}) + \sum_{i=1}^n f(\widehat{w_{i-1}w_i}).
\]
Note that the union $A_0$ of the two tilings is a $2$-tiling of $\S^1$.
Assume that the points  $x_1, x_2, \ldots, x_{n-1}$, and $y_1,y_2, \ldots, y_{n+1}$ are in this counterclockwise order in $\S^1$, and that $y_1 \in \widehat{x_1x_2}$.
Due to the possible existence of coinciding points in the above two sequences, we unite these sequences as a single sequence $v_1, v_2, \ldots, v_{2n}$ in such a way that the points are in this counterclockwise order in $\S^1$, $v_1=x_1$, and removing the $x_i$ (resp. $y_j$) from this sequence we obtain the sequence $y_1, \ldots, y_{n+1}$ (resp. $x_1, \ldots, x_{n-1}$). In the proof we regard this sequence as a cyclic sequence, where the indices are determined mod $2n$, and, with a little abuse of notation, we say that $\widehat{v_iv_j}$ \emph{covers} $\widehat{v_kv_l}$ only if $\widehat{v_kv_l} \subseteq \widehat{v_iv_j}$ and $i < k < l < j < i+2n$. 
Our main goal will be to modify the $2$-tiling $A_0$ in such a way that the value of $f$ does not decrease but the number of covering pairs strictly decreases.

Note that since $A_0$ is the union of two tilings consisting of $(n-1)$ and $(n+1)$ arcs, respectively, $A_0$ contains covering pairs. 
Assume that $\widehat{v_iv_j}$ covers $\widehat{v_kv_l}$. Then let $A_1$ denote the $2$-tiling of $\S^1$ in which $\widehat{v_iv_j}$ and $\widehat{v_kv_l}$
are replaced by $\widehat{v_iv_l}$ and $\widehat{v_kv_j}$. According to our conditions, $\sum_{S \in A_0} f(S) \leq \sum_{S \in A_1} f(S)$. Furthermore, as $A_0$ and $A_1$ are $2$-tilings, and $\widehat{v_kv_l}$ is already covered twice by the arcs considered in the modification, the number of covering pairs in $A_1$ is strictly less than in $A_0$. Repeating this procedure we obtain a $2$-tiling $A_t$ of $\S^1$ for which $\sum_{S \in A_0} f(S) \leq \sum_{S \in A_t} f(S)$ and which does not contain covering pairs. Then, $A_t$ decomposes into the two tilings $\{ \widehat{v_1,v_3}, \widehat{v_3v_5}, \ldots, \widehat{v_{2n-1}v_1} \}$ and $\{ \widehat{v_2,v_4}, \widehat{v_4v_6}, \ldots, \widehat{v_{2n}v_2} \}$, each of which contains exactly $n$ arcs. This proves the assertion for $\{ M_n \}$.

Now we prove the second part. Let $X$ be an $n$-element tiling of $\S^1$ such that $M_n = \sum_{S \in X} f(S)$. Assume that $X$ does not have $k$-fold rotational symmetries. For $i=1,2,\ldots, k$, let $X_i$ denote the rotated copy of $X$ by $\frac{2i\pi}{k}$ in counterclockwise direction. Then $Y= \bigcup_{i=1}^k X_i$ is a $k$-fold tiling of $\S^1$ with $k$-fold rotational symmetry, and $\sum_{S \in Y} f(S) = k \sum_{S \in X} f(S)$. Since $X$ has no $k$-fold rotational symmetry, $Y$ contains covering pairs, and we may apply the argument in the previous paragraph.
\end{proof}

We remark that an analogous proof yields Lemma~\ref{lem:functional_2}, the proof of which we leave to the reader.

\begin{Lemma}\label{lem:functional_2}
Let $f : \mathcal{S} \to \Re$ be a bounded function with $f(\widehat{pp})=0$ for all $p \in \S^1$. For any integer $n \geq 3$, let
\[
m_n = \inf \{ \sum_{S \in X} f( S ) : X \subset \mathcal{S} \hbox{ is a tiling of } \S^1 \hbox{ with } |X| = n \}.
\]
If for any $\widehat{x_2x_3} \subset \widehat{x_1x_4}$, we have
\[
f(\widehat{x_1x_3})+f(\widehat{x_2x_4}) \leq f(\widehat{x_1x_4})+f(\widehat{x_2x_3}),
\]
then the sequence $\{ m_n \}$ is convex.
Furthermore, if in addition, there is some positive integer $k$ such that $k | n$, and $f$ has $k$-fold rotational symmetry, and there is an $n$-element tiling  $X$ of $\S^1$ such that $m_n = \sum_{S \in X} f(S)$  then there is a tiling $X'$ of $\S^1$ with $k$-fold rotational symmetry such that $m_n = \sum_{S \in X' } f(S)$.
\end{Lemma}

In the next lemma, by the partial derivatives $(\partial_p f) (\widehat{p_0q_0})$ (resp. $(\partial_q f) (\widehat{p_0q_0})$) of the function $f(\widehat{pq})$ at $\widehat{p_0q_0}$, we mean the derivative of the function $f(\widehat{p(t)q_0}$) (resp. $f(\widehat{q(t)p_0})$) at $t=0$, where $p(t)$ (resp. $q(t)$) is the rotated copy of $p_0$ (resp. $q_0$) around $o$ by angle $t$ in counterclockwise direction.

\begin{Lemma}\label{lem:derivatives}  
Let $f : \mathcal{S} \to \Re$ be a bounded function with $f(\widehat{pp}) = 0$ for all $p \in \S^1$. Assume that for any $\widehat{p_0q_0} \in \S^1$, where $p_0 \neq q_0$, $(\partial_p \partial_q f)(\widehat{p_0q_0})$ is a continuous function of $\widehat{p_0q_0}$ in both variables.
Then, for any $x_1, x_2, x_3, x_4 \in \S^1$ in this counterclockwise order, we have
\[
f(\widehat{x_1x_3})+f(\widehat{x_2x_4}) \geq f(\widehat{x_1x_4})+f(\widehat{x_2x_3})
\]
if and only if $(\partial_p \partial_q f)(\widehat{p_0q_0}) \geq 0$ for all $p_0 \neq q_0$.
Similarly, for any $x_1, x_2, x_3, x_4 \in \S^1$ in this counterclockwise order, we have
\[
f(\widehat{x_1x_3})+f(\widehat{x_2x_4}) \leq f(\widehat{x_1x_4})+f(\widehat{x_2x_3})
\]
if and only if $(\partial_p \partial_q f)(\widehat{p_0q_0}) \leq 0$ for all $p_0 \neq q_0$.
\end{Lemma}

\begin{proof}
We prove only the first part. Assume that $(\partial_p \partial_q f)(\widehat{p_0q_0}) \geq 0$ for all $p_0 \neq q_0$. Let $\widehat{x_2x_3} \subset \widehat{x_1x_4}$.
Then, by the Newton--Leibniz Theorem we have
\[
0 \leq \int_{x_3}^{x_4} \int_{x_1}^{x_2} (\partial_p \partial_q f)(\widehat{p_0q_0}) \, d p_0 \, d q_0 = f(\widehat{x_2x_4})-f(\widehat{x_2x_3})-f(\widehat{x_1x_4})+f(\widehat{x_1x_3}).
\]
Furthermore, if we have $(\partial_p \partial_q f)(\widehat{p_0q_0}) < 0$ for some $p_0 \neq q_0$, then, by continuity and the same argument, there are some points $x_1,x_2$ and $x_3,x_4$ sufficiently close to $p_0$ and $q_0$, respectively, such that $\widehat{x_2x_3} \subset \widehat{x_1x_4}$, and $0 > f(\widehat{x_2x_4})-f(\widehat{x_2x_3})-f(\widehat{x_1x_4})+f(\widehat{x_1x_3})$.
\end{proof}

\begin{Remark}
The main tool of the proof of Theorems~\ref{thm:Dowker1} and \ref{thm:Dowker2} for $\hat{A}_n^C(K)$ with any fixed $K$ is to find a function $f : \mathcal{S} \to \Re$ such that the quantity $m_n$ in Lemma~\ref{lem:functional_2} coincides with $\hat{A}_n^C(K)$, and check that the conditions required by the lemma are satisfied. In the proof for $\hat{p}_n^C(K)$ we apply the same idea, but we use Lemma~\ref{lem:functional_1}. The proof for $\hat{P}_n^C(K)$ follows the same line, but we also use Lemma~\ref{lem:derivatives} together with Lemma~\ref{lem:functional_2}.
\end{Remark}

\section{Proof of Theorems~\ref{thm:Dowker1} and \ref{thm:Dowker2}}\label{sec:proofs}

Note that by Lemma~\ref{lem:closed} and Corollary~\ref{cor:closed}, it is sufficient to prove 
Theorem~\ref{thm:Dowker1} for any everywhere dense subset of $\KK_o$, and applying a similar consideration, we have the same for Theorem~\ref{thm:Dowker2}. Thus, we may assume that $C$ has $C^{\infty}$-class boundary and strictly positive curvature \cite{Schneider, Schmuckenschlaeger}. Furthermore, by Lemmas~\ref{lem:smoothisdense} and \ref{lem:defncontinuity}, we may assume that $K$ has $C^{\infty}$-class boundary, and the curvature of $\bd(K)$ at any point $p$ is strictly greater than the curvature of $\bd(C)$ at the point $q$ with the same outer normal as $p$.

\begin{Remark}\label{rem:longestchord}
Under the above conditions, for any points $p,q \in \bd (K)$, $[p,q]_C \setminus \{ p,q \} \subset \inter (K)$.
\end{Remark}

In the proof we identify $\S^1$ with the set $\Re / \{ 2k\pi : k \in \mathbb{Z}\}$. Let us parametrize $\bd(K)$ as the curve $\Gamma : \S^1 \to \Re^2$, where the outer unit normal vector at $\Gamma(\varphi)$ is $(\cos \varphi, \sin \varphi)$.
Then, for any two points $\Gamma(\varphi_1), \Gamma(\varphi_2)$ with $\varphi_1 < \varphi_2 < \varphi_1+2\pi$, let us denote the arc of $\Gamma$ connecting them in counterclockwise direction by $\Gamma|_{[\varphi_1,\varphi_2]}$. Furthermore, recall \cite[Corollary 3.13]{LNT2013}, stating that $K$ is the intersection of the translates of $C$ containing it. Thus, for any $\varphi \in [0,2\pi]$, there is a unique translate $x+C$ of $C$ containing $K$ with $\Gamma(\varphi) \in \bd (x+C)$.
We denote this translate by $C(\varphi)=x(\varphi)+C$, and call it the \emph{supporting $C$-disk} of $K$ at $\Gamma(\varphi)$ (see Figure~\ref{fig:regions}).

\begin{figure}[ht]
\begin{center}
 \includegraphics[width=0.9\textwidth]{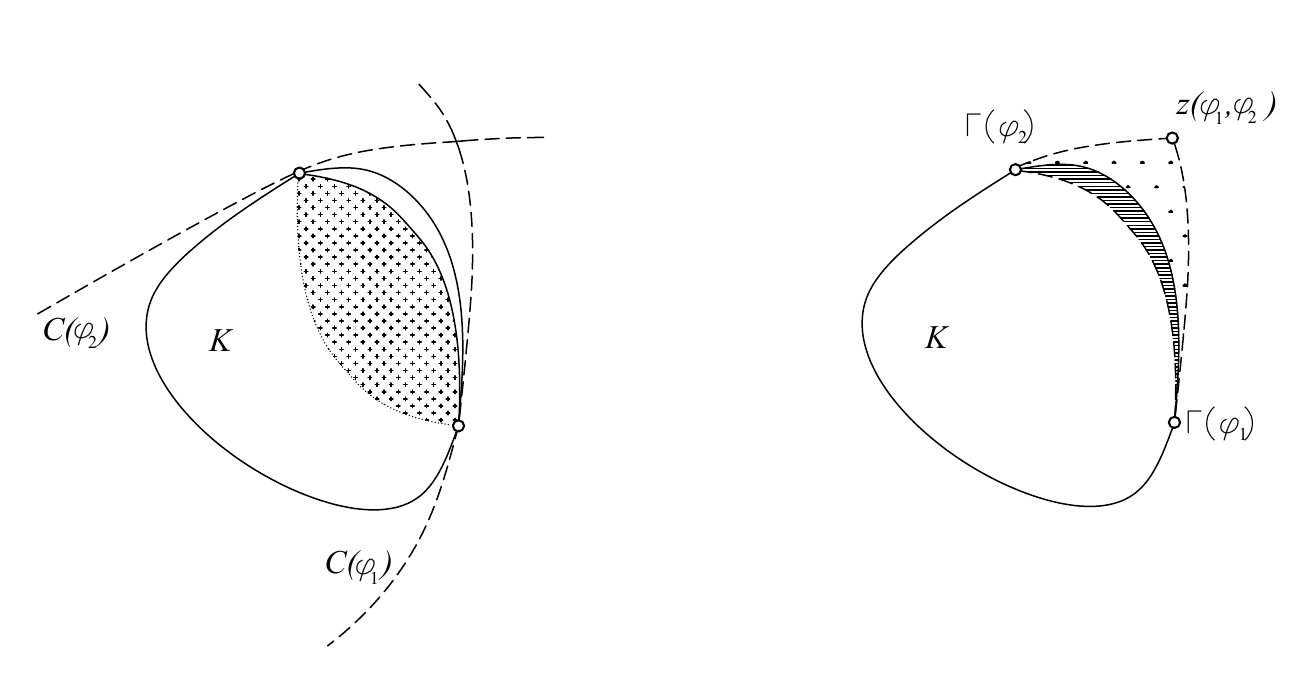}
 \caption{An illustration for the notation in the proof of Theorems~\ref{thm:Dowker1} and \ref{thm:Dowker2}. The left panel shows the $C$-spindle $[\Gamma(\varphi_1), \Gamma(\varphi_2)]_C$ (the region with little crosses), and the boundaries of the support disks $C(\varphi_1)$, $C(\varphi_2)$ (dashed lines). The right panel shows the regions $r(\varphi_1,\varphi_2)$ (with horizontal stripes) and $R(\varphi_1,\varphi_2)$ (with dots).}
\label{fig:regions}
\end{center}
\end{figure}

We define the following regions and quantities:

\begin{enumerate}
\item[(i)] $r(\varphi_1,\varphi_2)$ is the closure of the connected component of $K \setminus [\Gamma(\varphi_1), \Gamma(\varphi_2)]_C$ containing $\Gamma|_{[\varphi_1,\varphi_2]}$;
\item[(ii)] $R(\varphi_1,\varphi_2)$ is the closure of the connected component of $(C(\varphi_1) \cap C(\varphi_2) \setminus K)$ containing $\Gamma|_{[\varphi_1,\varphi_2]}$;
\end{enumerate}
\begin{enumerate}
\item[(1)] $p(\varphi_1,\varphi_2) = \perim_C(r(\varphi_1,\varphi_2)) - \arclength_C(\Gamma|_{[\varphi_1,\varphi_2]})$;
\item[(2)] $A(\varphi_1,\varphi_2) = \area(R(\varphi_1,\varphi_2))$;
\item[(3)] $P(\varphi_1,\varphi_2) = \perim_C(R(\varphi_1,\varphi_2)) - \arclength_C(\Gamma|_{[\varphi_1,\varphi_2]})$.
\end{enumerate}

We observe that since we have identified $\S^1$ with $\Re / \{ 2k\pi : k \in \mathbb{Z}\}$, the functions $p$, $A$ and $P$ are defined on the closed circle arcs of $\S^1$.

\subsection{The proof of Theorems~\ref{thm:Dowker1} and \ref{thm:Dowker2} for $\hat{A}_n^C(K)$}\label{area_circum}

Let $I[X] : \Re^2 \to \Re$ denote the indicator function of $X \subset \Re^2$. Then it can be seen directly that for any $\varphi_1 < \varphi_2 < \varphi_3 < \varphi_4 < \varphi_1+2\pi$, the function
\[
I[R(\varphi_1,\varphi_4)] + I[R(\varphi_2,\varphi_3)] - I[R(\varphi_1,\varphi_3)]- I[R(\varphi_2,\varphi_4)]
\]
has nonnegative values at every point. Thus, the conditions of Lemma~\ref{lem:functional_2} are satisfied for $A(\varphi_1,\varphi_2)$, implying the statement.

\subsection{The proof of Theorems~\ref{thm:Dowker1} and \ref{thm:Dowker2} for $\hat{p}_n^C(K)$}

Let $\varphi_1 < \varphi_2 < \varphi_3 < \varphi_4 < \varphi_1+2\pi$. Then, by Lemma~\ref{lem:quadrangle},
\[
\rho_{C}(\Gamma(\varphi_1),\Gamma(\varphi_4))+\rho_{C}(\Gamma(\varphi_2),\Gamma(\varphi_3)) \leq \rho_{C}(\Gamma(\varphi_1),\Gamma(\varphi_3))+\rho_{C}(\Gamma(\varphi_2),\Gamma(\varphi_4)).
\]
Thus, the conditions of Lemma~\ref{lem:functional_1} are satisfied for $p(\varphi_1,\varphi_2)$, implying our statement.

\subsection{The proof of Theorems~\ref{thm:Dowker1} and \ref{thm:Dowker2} for $\hat{P}_n^C(K)$}

By Lemmas \ref{lem:functional_2} and \ref{lem:derivatives}, it is sufficient to prove that for any $\varphi_1 < \varphi_2 < \varphi_1+\pi$, the function $\partial_{\varphi_1} \partial_{\varphi_2} P$ is a continuous nonpositive function; then the statement follows from Lemmas~\ref{lem:derivatives} and \ref{lem:functional_2} applied for $P(\varphi_1,\varphi_2)$.
In the remaining part of the subsection we prove this property.

For brevity, for any $\alpha < \beta < \alpha +2\pi$, we define $z(\alpha,\beta)$ as the intersection point of $\bd(C(\alpha))$ and $\bd(C(\beta))$ contained in the boundary of $R(\alpha,\beta)$.
First, observe that $P(\varphi_1,\varphi_2) = \rho_C(\Gamma(\varphi_1),z(\varphi_1,\varphi_2))+ \rho_C(z(\varphi_1,\varphi_2),\Gamma(\varphi_2))$. Clearly, since $C$ has $C^{\infty}$-class boundary, $\rho_C(\cdot,\cdot)$ is a $C^{\infty}$-class function, implying that $P(\varphi_1,\varphi_2)$ is $C^{\infty}$-class, and $\partial_{\varphi_1} \partial_{\varphi_2} P$ is continuous.

 Now, let $0 < | \Delta_1| , |\Delta_2 | \leq \varepsilon$ for some sufficiently small $\varepsilon > 0$, and set $p=z(\varphi_1,\varphi_2)$, $q_1=z(\varphi_1,\varphi_2+\Delta_2)$, $q_2 = z(\varphi_1 + \Delta_1,\varphi_2)$ and $q=z(\varphi_1+\Delta_1,\varphi_2+\Delta_2)$.

To prove the assertion, it is sufficient to prove that
\[
0 \geq \frac{1}{\Delta_1} \left( \frac{P(\varphi_1+\Delta_1,\varphi_2+\Delta_2)-P(\varphi_1+\Delta_1,\varphi_2)}{\Delta_2} - \frac{P(\varphi_1,\varphi_2+\Delta_2)-P(\varphi_1,\varphi_2)}{\Delta_2}\right) =
\]
\[
= \frac{1}{\Delta_1 \Delta_2} \left( P(\varphi_1+\Delta_1,\varphi_2+\Delta_2) -  P(\varphi_1+\Delta_1,\varphi_2) - P(\varphi_1,\varphi_2+\Delta_2) + P(\varphi_1,\varphi_2) \right).
\]

We do it in the case that $\Delta_1 < 0$ and $\Delta_2 > 0$, in the other cases a straightforward modification yields the assertion.
Note that in this case it is sufficient to show that
\[
\rho_C(p,q_1)+\rho_C(p,q_2) \leq \rho_C(q,q_1)+\rho_C(q,q_2).
\]
For $i=1,2$, let $v_i$ denote the tangent vector of $C(\varphi_i)$ at $p$ pointing `towards' $q_i$ in its boundary, and let $w_i$ denote the tangent vector of $\bd K$ at $\Gamma(\varphi_i)$ pointing towards $p$ in $\bd(C(\varphi_i))$.

\begin{figure}[ht]
\begin{center}
 \includegraphics[width=0.5\textwidth]{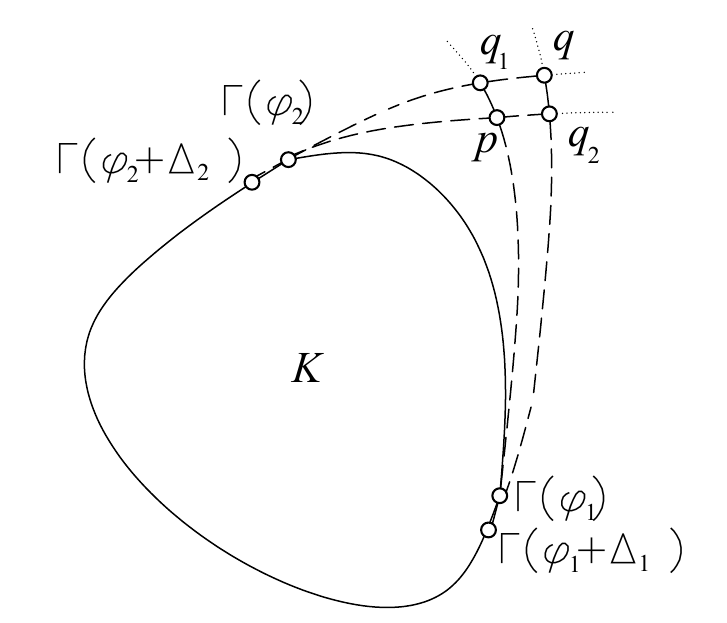}
 \caption{Notation for the proof of Theorems~\ref{thm:Dowker1} and \ref{thm:Dowker2} for $\hat{P}_n^C(K)$.}
\label{fig:circumperim}
\end{center}
\end{figure}

\begin{Lemma}\label{lem:translation}
Let $C(\varphi)= x(\varphi)+C$. Then $\lim_{\Delta \to 0 \pm 0} \frac{x(\varphi+\Delta)-x(\varphi)}{|x(\varphi+\Delta)-x(\varphi)|} = \pm v$ for any value of $\varphi$, where $v$ is the unit tangent vector of $\bd(K)$ at $\Gamma(\varphi)$ pointing in the positive direction.
\end{Lemma}

\begin{proof}
Let $\Theta(\varphi)$ denote the point of $\bd(C)$ with outer unit normal vector $(\cos \varphi, \sin \varphi)$. Then $x(\varphi)=\Gamma(\varphi)-\Theta(\varphi)$ and
more generally,
\[
x(\varphi+\Delta)-x(\varphi) = \left( \Gamma(\varphi+\Delta)- \Gamma(\varphi) \right) - \left( \Theta(\varphi+\Delta)- \Theta(\varphi) \right). 
\]
Note that $\lim_{\Delta \to 0 \pm 0} \frac{\Gamma(\varphi+\Delta)- \Gamma(\varphi)}{|\Gamma(\varphi+\Delta)- \Gamma(\varphi)|} = \lim_{\Delta \to 0 \pm 0} \frac{\Theta(\varphi+\Delta)- \Theta(\varphi)}{|\Theta(\varphi+\Delta)- \Theta(\varphi)|} = \pm v$, and, by the choice of the parametrization of $\Gamma$ and $\Theta$, $\lim_{\Delta \to 0} \frac{|\Theta(\varphi+\Delta)- \Theta(\varphi)|}{|\Gamma(\varphi+\Delta)- \Gamma(\varphi)|} = \frac{\kappa_{\Gamma}(\varphi)}{\kappa_{\Theta}(\varphi)}$, where $\kappa_{\Gamma}(\varphi)$ and $\kappa_{\Theta}(\varphi)$ denote, the curvature of $\Gamma$ and $\Theta$ at $\Gamma(\varphi)$ and $\Theta(\varphi)$,respectively. Thus, the assertion follows from our assumption that $\kappa_{\Theta}(\varphi) \neq \kappa_{\Gamma}(\varphi)$.
\end{proof}

By Lemma~\ref{lem:Cgons}, $C(\varphi_1) \cap C(\varphi_2)$ is the $C$-spindle of $p$ and another point, which we denote by $p'$.
By convexity, the tangent vectors of $\bd(C(\varphi_1))$ pointing in counterclockwise direction, turn in counterclockwise direction from $p$ to $p'$. Thus, the directions of the vectors $v_2, w_1, v_1$ are in this order in counterclockwise orientation, and the same holds for the vectors $v_2, w_2, v_1$.

For $i=1,2$, let $C_i^* = C(\varphi_i+\Delta_i)=y_i + C(\varphi_i)$, and set $Q=C_1^* \cap C_2^*$. Then, by Lemma~\ref{lem:translation}, if $\Delta_i$ is sufficiently small, we have that the vectors $y_1,y_2$ are between $v_1$ and $v_2$ according to counterclockwise orientation. 

Consider the translate $C_i'$ of $C(\varphi_i)$ by $q_i-p$. The boundary of this translate contains $q_i$, and $v_i$ is a tangent vector of $C_i'$ at $q_i$. Thus, if $q' = q_1+q_2-p$ (i.e. $q'$ is the unique point for which $p,q_1,q',q_2$ are the vertices of a parallelogram in this counterclockwise order), then $q'$ lies in the boundary of both $C_1'$ and $C_2'$. On the other hand, by our observation about the tangent lines, if $\Delta_i$ are sufficiently small, then $q'$ is contained in $Q$. By symmetry, $\rho_C(p,q_1) = \rho_C(q',q_1)$ and $\rho_C(p,q_2) =\rho_C(q',q_2)$, and thus, the required inequality follows from the remark after Definition~\ref{defn:Cperim}.

\section{Proof of Theorem~\ref{thm:counter1}}\label{sec:counter1}

We prove the statement in three steps.
\begin{itemize}
\item In the first step, we construct some $C_1 \in \KK_o$ and a $C_1$-convex disk $K_1 = \conv_{C_1} \{ p_1, p_2, p_5 p_6 \}$ such that for any $x \in K_1$, we have
$\area([p_1,p_6]_{C_1})+\area(\conv_{C_1} \{ p_1,p_2,p_5,p_6 \}) > 2 \area(\conv_{C_1}\{p_1, p_6, x \})$.
\item In the second step, for every $n \geq 4$ we extend the above example to construct some $C_n \in \KK_0$ and a $C_n$-convex disk $K_n$ such that $\hat{a}_{n-1}^{C_n}(K_n) + \hat{a}_{n+1}^{C_n}(K_n) > 2 \hat{a}_{n}^{C_n}(K_n)$.
\item In the third step, we show that a suitable small linear image of the essential parts of $\bd (C_n)$ and $\bd(K_n)$ from Step 2 can be embedded in the boundary of any $C \in \KK_o$ to obtain a slightly modified convex disk, denoted by $C'$, such that a similarly modified linear image $K'$ of $K_n$ will be $C'$-convex, and satisfies $\hat{a}_{n-1}^{C'}(K') + \hat{a}_{n+1}^{C'}(K') > 2 \hat{a}_{n}^{C'}(K')$.
\end{itemize}

For brevity, for any points $z_1,z_2, \ldots, z_k \in \Re^2$, we set $[z_1,z_2,\ldots,z_k] = \conv \{ z_1,z_2,\ldots, z_k \}$ and
$[z_1,z_2,\ldots,z_k]_C = \conv_C \{ z_1,z_2,\ldots, z_k \}$.

\noindent
\textbf{Step 1}.\\
Let us fix a Cartesian coordinate system, and consider the points $p_1=(0,-1-t)$, $p_2=(2.1,-0.9-t)$, $p_3=(t+2,-1)$, $p_4=(t+2,1)$, $p_5=(2.1, 0.9+t)$, $p_6=(0,1+t)$, $q_1=(t,-1)$, $q_2=(t,1)$, $q_3=(-t,1)$ and $q_4=(-t,-1)$ (see Figure~\ref{fig:counter_1}). Here, we note that $p_4-q_2=p_3-q_1=(2,0)$, but $p_2-p_1=(2.1,0.1) \neq (2,0)$ and $p_5-p_6=(2.1, -0.1) \neq (2,0)$. In addition, we note that the segments $|p_2,p_3], [p_4,p_5]$ are parallel to $[p_1,q_1], [q_2,p_6]$, respectively. In the construction we assume that $t$ is a sufficiently large positive value. 
We define the hexagon $H= [p_1,q_1,q_2,p_6,q_3,q_4]$ and the octagon $K_1 = [p_1,p_2,\ldots,p_6,q_3,q_4]$. Note that $H \subset K_1$, and set $G = \bd(K_1) \setminus \bd(H)$, and $G'=\bd(K_1) \cap \bd(H)$. In the following, $D_1$ denotes the Euclidean diameter of $K_1$.

\begin{figure}[ht]
\begin{center}
 \includegraphics[width=0.4\textwidth]{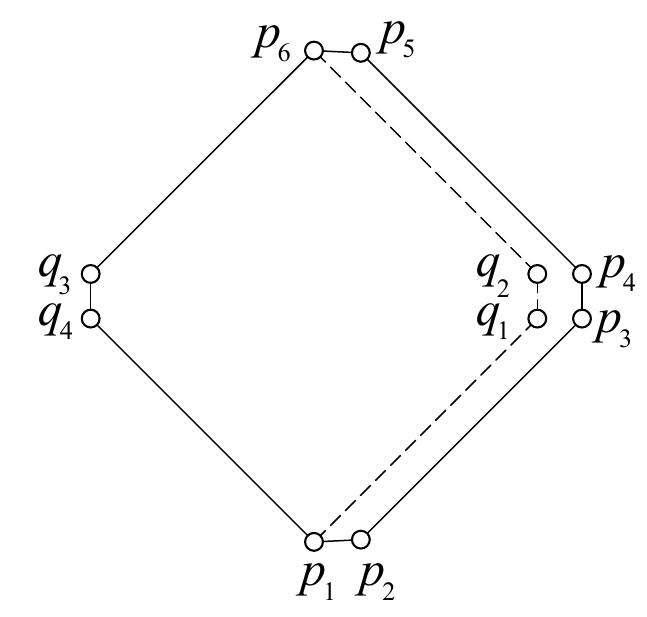}
 \caption{The hexagon $H$ and the octagon $K_1$. In the illustration, $t=10$.}
\label{fig:counter_1}
\end{center}
\end{figure}

We define $C_1$ as an $o$-symmetric convex $14$-gon with vertices $x_1,x_2,\ldots,x_{14}$ in counterclockwise order such that
\begin{enumerate}
\item[(a)] $x_1$ and $x_8$ are on the negative and the positive half of the $y$-axis, respectively;
\item[(b)] $C_1$ is symmetric to both coordinate axes;
\item[(c)] the sides $[x_1,x_2]$, $[x_2,x_3]$, $[x_3,x_4]$, $[x_4,x_5]$ are parallel to $[p_1,p_2]$, $[p_1,p_3]$, $[p_2,p_3]$ and $[p_3,p_4]$, respectively;
\item[(d)] we have $||x_2-x_1||, ||x_3-x_2||, ||x_4-x_3|| > D_1$, and $||x_5-x_4||=2$, i.e. $[x_4,x_5]$ is a translate of $[p_3,p_4]$.
\end{enumerate}

Note that by our conditions, for any two point $u,v \in G$, each of the two $C_1$-arcs in the boundary of $[u,v]_{C_1}$ consists of translates of subsets of at most two consecutive sides of $C_1$, or they contain translates of $[x_4,x_5]$ and possibly translates of subsets of the sides $[x_3,x_4]$ and $[x_5,x_6]$.
In particular, $[p_1,p_6]_{C_1} = H$ directly follows from Lemma~\ref{lem:Cgons} and its geometric interpretation mentioned earlier.

We estimate $\area([p_1,q,p_6]_{C_1})$ for any $q \in G$ with nonnegative $y$-coordinate. In the following $\bar{p}=(0,t+2)$ denotes the midpoint of $[p_3,p_4]$.
Our general procedure to find $[p_1,q,p_6]_{C_1}$ for some point $q$ will be based on the following observation. By Lemma~\ref{lem:Cgons}, each pair of points is connected by a translate of a $C_1$-arc. For $p_1,p_6$, this translate is clearly $G'$. To find the $C_1$-arc whose translate connects $p_1,q$, we find the two chords of $C_1$ that are translates of $[p_1,q]$; the position of these chords is determined by the direction of the segment $[p_1,q]$. The chords correspond to two $C_1$-arcs shorter than half of $\bd (C_1)$, and we choose the one containing points with positive $x$-coordinates. We repeat the same procedure for $p_6$ and $q$.

\emph{Case 1}: $q \in [\bar{p},p_4]$. Then $\bd([p_1,q,p_6]_{C_1})$ consists of $G'$, parts of the segments $[p_1,p_3]$ and $[p_4,p_6]$, and two segments with $q$ as an endpoint, parallel to $[p_2,p_3]$ and $[p_4,p_5]$, respectively.

\begin{Claim}\label{cl:whythis}
The quantity $\area([p_1,q,p_6]_{C_1})$ is maximal if $q=\bar{p}$.
\end{Claim}
\begin{proof}
 Let $[q,u_1]$ and $[q,u_2]$ denote the two sides of $\bd([p_1,q,p_6]_{C_1})$ with $q$ as a vertex. We choose the indices in such a way that $u_1 \in [p_1,p_3]$ and $u_2 \in [p_6,p_4]$. Note that since $q \in [\bar{p},p_4]$, the Euclidean length of $[u_1,q]$ is not smaller than that of $[u_2,q]$. Thus, moving $q$ slightly towards $p_3$ (or $p_4$) increases (or decreases) $\area([p_1,q,p_6]_{C_1})$.
\end{proof}
Now, by Claim~\ref{cl:whythis},
\[
\area([p_1,q,p_6]_{C_1}) \leq \area([p_1,\bar{p},p_6]_{C_1}) = \area(H)+\frac{3}{2}t + 3 < \area(H) + 2t-0.1 = \frac{1}{2} \left( \area(H)+\area(K_1) \right) - 2,
\]
if $t$ is sufficiently large.

\noindent
\emph{Case 2}: $q \in [p_4,p_5]$. Assume that the $x$-coordinate of $q$ is at least $t+1$. Then the curve $\bd([p_1,q,p_6]_{C_1})$ consists of $G'$, a segment $[p_1,v_1]$ containing $[p_1,q_1]$, a segment $[v_1,q]$ parallel to $[p_3,p_4]$, a segment $[q,v_2]$ parallel to $[p_4,p_6]$, and a subset $[v_2,p_6]$ of $[p_5,p_6]$. Observe that as $t \to \infty$, $\frac{||v_2-q||}{||v_1-q||} \to \infty$, implying that if $t$ is sufficiently large, then moving $q$ towards $p_5$ increases $\area([p_1,q,p_6]_{C_1})$. Thus, in this case the area is maximal if the $x$-coordinate of $q$ is equal to $t+1$.

Now, consider the case that the $x$-coordinate of $q$ is at most $t+1$. Then $\bd([p_1,q,p_6]_{C_1})$ consists of $G'$, the curve $[p_1,v_1'] \cup [v_1',v_2'] \cup [v_2',q]$, where $v_1'=(t+1,0)$ and $v_2'=(t+1,2)$, a segment $[q,v_3']$ parallel to $[p_4,p_6]$, and a subset $[v_3',p_6]$ of $[p_5,p_6]$. From this description, it readily follows that $\area([p_1,q,p_6]_{C_1})$ is maximal if $q=p_5$. Thus, in Case 2 we have
\[                                                                                                                  
\begin{aligned}
\area([p_1,q,p_6]_{C_1}) \leq \area([p_1,p_5,p_6]_{C_1}) = \\
= \area(H)+ \area([q_2,p_4,p_5,p_6)=\frac{1}{2} \left( \area(H)+\area(K_1) \right) - 2
\end{aligned}
\]

\emph{Case 3}: $q \in [p_5,p_6]$. Then $\bd([p_1,q,p_6]_{C_1})$ consists of $G'$, a segment parallel to $[q_2,p_6]$ and ending at $q$, a segment containing $[p_1,q_1]$ as a subset, and a translate of $[p_3,p_4]$. Thus, in this case $\area([p_1,q,p_6]_{C_1})$ is maximal if $q=p_5$, and we have
\[
\area([p_1,q,p_6]_{C_1}) \leq \area([p_1,p_5,p_6]_{C_1}) = \frac{1}{2} \left( \area(H)+\area(K_1) \right) - 2.
\]

Combining our results, if $t$ is sufficiently large, for any $q,q' \in G$
\begin{multline}\label{eq:important}
\area([p_1,q,p_6]_{C_1}) + \area([p_1,q',p_6]_{C_1}) \leq \area(H)+\area(K_1) - 4 <\\
< \area(H)+\area(K_1) = \area([p_1,p_6]_{C_1})+\area([p_1,p_2,p_5,p_6]_{C_1}),
\end{multline}
where we used the observation that $[p_1,p_2,p_5,p_6]_{C_1} = K_1$. To see this observation, we note that the $C_1$-arc connecting any consecutive pair of points in the (cyclic) sequence $p_1,p_2,p_5,p_6$ is contained in $\bd (K_1)$.

In the remaining part of the construction, we fix $t$ in such a way that (\ref{eq:important}) is satisfied. Before proceeding to Step 2, we emphasize that to guarantee the inequality $\area([p_1,p_6]_{C_1})+\area(\conv_{C_1} \{ p_1,p_2,p_5,p_6 \}) > 2 \area(\conv_{C_1}\{p_1, p_6, x \})$, we had to consider only chords of $C_1$ connecting two points $\bd(C_1)$ with nonnegative coordinates. This yields that if we replace $C_1$ with an $o$-symmetric convex disk whose boundary contains a translate of the part of $\bd(C_1)$ with nonnegative $x$-coordinates, then the above inequality remains true.

\textbf{Step 2}.\\
In the next step, based on Step 1, for any $n \geq 4$, we construct some $C_n \in \KK_o$ and a $C_n$-convex disk $K_n$ such that 
\begin{equation}\label{eq:step2}
\hat{a}_{n-1}^{C_n}(K_n) + \hat{a}_{n+1}^{C_n}(K_n) > 2 \hat{a}_{n}^{C_n}(K_n).
\end{equation}

For easier understanding, first we present the construction for $n=4$.
Let $p_7 = (-s,0)$, where $s$ is sufficiently large, and set $K_4 = \conv (K_1 \cup \{ p_7 \})$ (see Figure~\ref{fig:counter_2}). Let $D_4$ denote the Euclidean diameter of $K_4$, and let $C^+_1$ (resp. $C^-_1$) denote the set of the points of $\bd(C_1)$ with nonnegative (resp. nonpositive) $x$-coordinates.
We define $C_4$ as follows:
\begin{enumerate}
\item[(a)] $C_4$ is symmetric to both coordinate axes.
\item[(b)] $\bd (C_4)$ contains some translates $u+ C^+_1$ and $-u+C^-_1$, where $u$ points in the direction of the positive half of the $x$-axis. We set $w_3=u+x_1$.
\item[(c)] In addition to the above two translates, $\bd(C_4)$ consists of segments $[w_1,w_2]$, $[w_2,w_3]$ and their reflections about one or both of the coordinate axes, such that $[w_1,w_2]$, $[w_2,w_3]$ are parallel to $[p_6,p_7]$ and $[p_5,p_7]$, respectively, and $||w_1-w_2||, ||w_2-w_3|| > D_4$.
\end{enumerate}
We remark that if $s$ is sufficiently large, in particular if the slope of $[p_6,p_7]$ is smaller than that of $[p_1,p_2]$, then there is some $C_4 \in \KK_o$ satisfying the above conditions. Since, among other things, the inequality $||w_1-w_2|| > D_4$ and the properties of the construction in Step 1 imply that $K_4$ is an intersection of some translates of $C_4$, we have that $K_4$ is $C_4$-convex.

\begin{figure}[ht]
\begin{center}
 \includegraphics[width=0.7\textwidth]{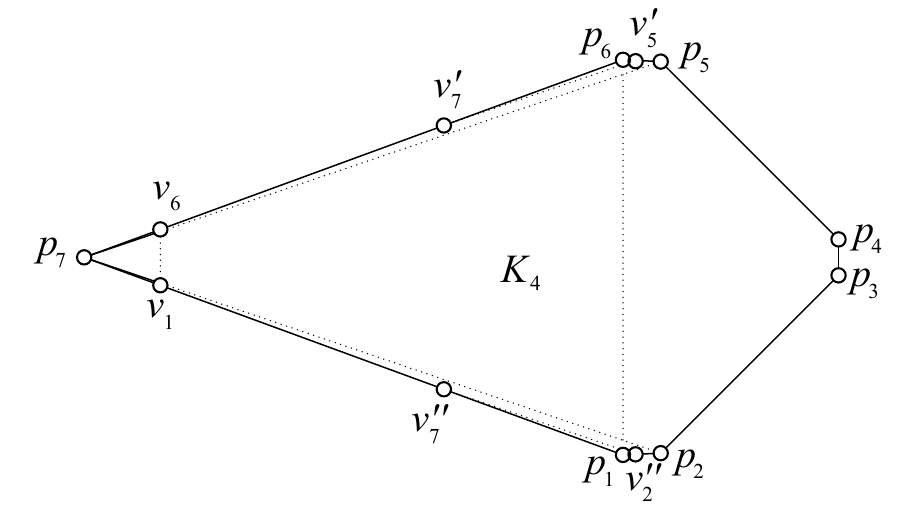}
 \caption{The $C_4$-convex disk $K_4$ with $t=10$ and $s=20$.}
\label{fig:counter_2}
\end{center}
\end{figure}

In the following, let $Q_4(s) = [z_1,z_2,z_3,z_4]_{C_4}$ denote a maximal area $C_4$-$4$-gon inscribed in $K_4$, for a fixed $s$.
Let $H'=\conv (H \cup \{ p_7 \}) = [p_1,p_6,p_7]_{C_4}$ and observe that $K_4 = [p_1,p_2,p_5,p_6,p_7]_{C_4}$. Then, to show the inequality in (\ref{eq:step2}), it is sufficient to show that $\area(H')+\area(K_4) > 2 \area(Q_4(s))$ for some value of $s$. 
Let $Q = [p_1,p_5,p_6,p_7]_{C_4}$. By the consideration in Step 1, we have that $\area(Q) = \frac{1}{2} (\area(H')+\area(K_4))-2$. Thus, we have $\area(Q_4(s)) \geq \frac{1}{2} (\area(H')+\area(K_4))-2$ for every $s$.

Let us define the points $v_1$ and $v_6$ as the images of $p_1$ and $p_6$, respectively, under the homothety with center $p_7$ and homothety ratio $\frac{1}{\sqrt[3]{s}}$. An elementary computation shows that then $v_1 = \left( -\left(1-\frac{1}{\sqrt[3]{s}} \right)s,  -\frac{1+t}{\sqrt[3]{s}}\right) \in [p_1,p_7]$ and  $v_6 = \left( -\left(1-\frac{1}{\sqrt[3]{s}} \right)s,  \frac{1+t}{\sqrt[3]{s}}\right) \in [p_6,p_7]$. Note that since $||v_6-v_1|| = \frac{2(1+t)}{\sqrt[3]{s}} < 2$ if $s$ is sufficiently large, and $\bd(C_4)$ contains two vertical segments of length $2$, we may assume that $[v_1,v_6]_{C_4} = [v_1,v_6]$. This yields that there is a translate $a+C$ of $C$ that contains $K_4 \setminus [v_1,p_7,v_6]$ and does not overlap $[v_1,p_7,v_6]$. Thus, if $[v_1,p_7,v_6]$ contains none of the $z_i$, then $Q_4(s) \subseteq (a+C) \cap K_4$ implies that $Q_4(s) \subseteq K_4 \setminus [v_1,p_7,v_6]$. From this it follows that in this case $\area(Q_4(s)) \leq \area(K_4) - \area([v_1,p_7,v_6]) = \area(K_4) - 2 \sqrt[3]{s}(1+t) < \frac{1}{2} (\area(H')+\area(K_4))-2$; a contradiction. Consequently, in the following we may assume that one of the $z_i$, say $z_4$, is a point of $T_7(s)=[v_1,p_7,v_6]$.

Let $v'_5$ and $v'_7$ be the images of $p_5$ and $p_7$, respectively, under the homothety with center $p_6$ and ratio $\frac{1}{\sqrt[3]{s}}$. Note that since the side $[w_2,w_3]$ of $C_4$ is parallel to $[v'_5,v'_7]$ and $||w_2-w_3|| > D_4 \geq ||p_5-p_7|| > ||v'_5 - v'_7||$, we have $[v_5',v_7']_{C_4}= [v_5',v_7']$, and, as in the previous paragraph, if no $z_i$ is contained in $[v_5',v_7',p_6]$, then $\area(Q_4(s)) \leq \area(K_4) - \area([v_5',v_7',p_6])$. To compute $\area([v_5',v_7',p_6])$, we note that $\area([p_5,p_7,p_6])$ is equal to one half of the absolute value of the determinant with $p_7-p_6, p_5-p_6$ as columns. Thus an elementary computation shows that
\[
\area([p_5,p_7,p_6]) = \frac{1}{2} \left| \begin{array}{cc} -s & 2.1 \\ -1-t & -0.1 \end{array}\right| = \frac{1}{2}(2.1+2.1t+0.1s).
\]
This yields that $\area([v_5',v_7',p_6]) = \frac{\area([p_5,p_7,p_6])}{\sqrt[3]{s^2}} > 0.05 \sqrt[3]{s}$, implying that if no $z_i$ is contained in $[v_5',v_7',p_6]$ and $s$ is sufficiently large, then $\area(Q_4) < \area(Q)$. Hence, we may assume that some $z_i$, say $z_3$, is an element of $T_6(s)=[v_5',p_6,v_7']$.

We obtain similarly that if $s$ is sufficiently large, some $z_i$, say $z_1$, is contained in the triangle $T_1(s)=[v_7'',p_1,v_2'']$, where $v_7''$ and $v_2''$ are the images of $p_7$ and $p_2$, respectively, under the homothety with center $p_1$ and ratio $\frac{1}{\sqrt[3]{s}}$.
Now, we estimate $\area(Q_4(s))$ as $s \to \infty$. Clearly,$\area(Q_4(s)) \leq \area([p_6,p_7,p_1]) + \area(Q_4(s) \cap K_+)$, where $K_+=\cl (K_4 \setminus [p_1,p_7,p_6]) = [p_1,p_2,\ldots,p_6]$. On the other hand, as $s \to \infty$, we have $z_3 \notin K_+$ or $z_3 \to p_6$, and similarly, $z_1 \notin K_+$ or $z_1 \to p_1$. Thus, by the argument in Step 1, for any $0 < \varepsilon < 4$, the inequality
\[
2\area(Q_4(s)) \leq 2 \area(Q)+\varepsilon = \area(H')+\area(K_4)-4+ \varepsilon < \area(H')+\area(K_4)
\]
holds for every sufficiently large value of $s$.

Next, we show how to modify the previous construction for arbitrary $n > 4$.
Consider an arbitrary convex $(n-1)$-gon $P'=[p_6,p_7', \ldots,p_{n+3}',p_1]$, with $[p_6,p_7'], \ldots, [p_{n+3}',p_1]$ as sides, that is separated from $p_2,p_5$ by the line through $[p_1,p_6]$, and its angles at $p_1$ and $p_6$ are acute. Let $h$ denote an orthogonal axial affinity, with the line through $[p_1,p_6]$ as axis, such that its ratio is a sufficiently large positive number. Note that since the $x$-coordinates of $p_1, p_6$ are $0$, this affinity is defined by $h((x,y))=(s x, y))$ for some sufficiently large ratio $s > 0$. Set $P=h(P')$ and for any integer $7 \leq i \leq n+3$, $p_i=h(p_i')$. Then we define $K_n$ as $K_n = [p_1,p_2,\ldots,p_{n+3}]$. We denote by $p_k$ the leftmost vertex of $K_n$, i.e. the vertex with smallest $x$-coordinate, and assume that the two neighbors of $p_k$ have different $x$-coordinates. We denote the (Euclidean) diameter of $K_n$ by $D_n$.

We define $C_n$ as follows:
\begin{enumerate}
\item[(a)] $\bd (C_n)$ contains some translates $u+ C^+_1$ and $-u+C^-_1$, where $u$ points in the direction of the positive half of the $x$-axis.
\item[(b)] In addition to the above two translates, $\bd(C_n)$ consists of $(2n-3)$ segments parallel to one of $[p_6,p_7],[p_7,p_8],\ldots,[p_{n+3},p_1]$ or to one of $[p_5,p_7], [p_6,p_8], \ldots, [p_{n+3}, p_2]$, where the length of each of these segments is greater than $D_n$.
\end{enumerate}

As for $n=4$, if $s$ is sufficiently large (i.e. the slope of $[p_6,p_7]$ is smaller than that of $[p_1,p_2]$, and the slope of $[p_{n+3},p_1]$ is smaller than that of $[p_5,p_6]$), then there is some $C_n \in \K_o$ satisfying the above conditions, and the property in (b) implies that $K_n$ is an intersection of translates of $C_n$, that is, $K_n$ is $C_n$-convex.

Let us define $H'=\conv (H \cup \{ p_7, p_8, \ldots, p_{n+3} \}) = [p_1,p_6,p_7,\ldots,p_{n+3}]_{C_n}$ and observe that $K_n = [p_1,p_2,p_5,p_6,\ldots,p_{n+3}]_{C_n}$. This implies that $H'$ is a $C_n$-$(n-1)$-gon and $K_n$ is a $C_n$-$(n+1)$-gon. Now, let $Q_n(s)$ denote a maximal area $C_n$-$n$-gon contained in $K_n$. For simplicity, we set $p_{n+4}=p_1$ and $p_{n+5}=p_2$, and for any integer $6 \leq i \leq n+4$ , we denote by $T_i$ the homothetic image of the triangle $[p_{i-1},p_i,p_{i+1}]$ under the homothety with $p_i$ as center and with ratio $\frac{1}{\sqrt[3]{s}}$. Then, similarly like in the case $n=4$ and using the fact that $C_n$ has a side parallel to $[p_{i-1},p_{i+1}]$ longer than $D_n$, one can show that $T_i$ contains a vertex of $Q_n(s)$ if $s$ is sufficiently large. Thus, as in the previous consideration, we have that for sufficiently large values of $s$,
\[
2\area(Q_n(s)) < \area(H')+\area(K_4),
\]
implying the inequality in (\ref{eq:step2}).

Let $C^+_n$ and $C^-_n$, we denote the parts of $\bd(C_n)$ contained in the closed half planes $\{ x \geq 0\}$ and $\{ x \leq 0\}$, respectively.
Before proceeding to the final step, we note that to guarantee the required inequality in (\ref{eq:step2}) we only used the properties of the arcs of $C_n$ \emph{entirely} contained in $C^+_n$ or $C^-_n$. Thus, if $C_n'$ is any $o$-symmetric plane convex body containing $C^+_n$ and $C^-_n$ in its boundary, then we have
$\hat{a}_{n-1}^{C_n'}(K_n) + \hat{a}_{n+1}^{C_n'}(K_n) > 2 \hat{a}_{n}^{C_n'}(K_n)$.
We formulate this observation in Claim~\ref{cl:count1_core}.

\begin{Claim}\label{cl:count1_core}
For any $n \geq 4$, there is some $C_n \in \KK_o$ and a $C_n$-convex disk $K_n$ such that if any $C_n' \in \KK_o$ contains $C_n^+$ and $C_n^-$ in its boundary, where by $C^+_n$ and $C^-_n$, we denote the parts of $\bd(C_n)$ contained in the closed half planes $\{ x \geq 0\}$ and $\{ x \leq 0\}$, respectively, then $K_n$ is $C_n'$-convex, and
\[
\hat{a}_{n-1}^{C_n'}(K_n) + \hat{a}_{n+1}^{C_n'}(K_n) > 2 \hat{a}_{n}^{C_n'}(K_n).
\]
\end{Claim}

\noindent
\textbf{Step 3}.\\
Now we prove Theorem~\ref{thm:counter1}. Let $n \geq 4$. Recall that $\KK^n_a$ denotes the elements $C$ of $\KK_o$ such that for any $C$-convex disk $K$, we have $\hat{a}_{n-1}^{C}(K) + \hat{a}_{n+1}^{C}(K) \leq 2 \hat{a}_{n}^{C}(K)$, and set $\overline{\KK}^n_a = \KK_o \setminus \KK^n_a$.
Observe that by Lemma~\ref{lem:closed}, $\overline{\KK}^n_a$ is open. We show that it is everywhere dense in $\KK_o$.

Let $C$ be an arbitrary element of $\KK_o$ and let $\varepsilon > 0$. Note that for any nondegenerate linear transformation $h : \Re^2 \to \Re^2$, $K$ is $C$-convex if and only if $h(K)$ is $h(C)$-convex, and for any $n \geq 4$, if $K$ is $C$-convex, then $\hat{a}_n^C(K) = \hat{a}_n^{h(C)}(h(K))$.
Thus, without loss of generality, we may assume that there are vertical supporting lines of $C$ meeting $\bd(C)$ at some points $\pm p$ of the $x$-axis. We choose our notation such that $p$ is on the positive half of the axis.

Consider the convex disk $C_n \in \KK_0$ in Claim~\ref{cl:count1_core}. Let us define the nondegenerate linear transformation $h_{\lambda, \mu} : \Re^2 \to \Re^2$ by $h_{\lambda,\mu}(x,y)=(\lambda x, \mu y)$. Then, if we choose suitable sufficiently small values $\mu, \lambda > 0$, then there is a translate $C^+$ of $h_{\lambda,\mu}(C^+_n)$, and an $o$-symmetric convex disk $C'$ containing $C^+$ in its boundary such that $C^+ \subset (C+ \varepsilon B^2) \setminus C$, and $C \subset C'$. Then $C' \cap (C+ \varepsilon B^2) \in \KK_o$ contains translates of $h_{\lambda,\mu}(C^+_n)$ and $h_{\lambda,\mu}(C^-_n)$ in its boundary, the Hausdorff distance of $C$ and $C'$ is at most $\varepsilon$, and, if we set $K'=h_{\lambda,\mu}(K_n)$, by Claim~\ref{cl:count1_core} we have
\[
\hat{a}_{n-1}^{C'}(K') + \hat{a}_{n+1}^{C'}(K') > 2 \hat{a}_{n}^{C'}(K').
\]
Thus, $\overline{\KK}^n_a$ is everywhere dense, which immediately yields that $\bigcap_{n=4}^{\infty} \overline{\KK}^n_a$ is residual, implying Theorem~\ref{thm:counter1}. 

\section{Open questions}\label{sec:remarks}

In the remaining part of the paper, we denote the set $[1,\infty) \cup \{ \infty \}$ by $[1,\infty]$.

Let $p,q \in [1,\infty]$ satisfy the equation $\frac{1}{p} + \frac{1}{q} = 1$. For any $K, L \in \KK$, G. Fejes T\'oth \cite{GFT} introduced
the \emph{weighted area deviation} of $K,L$ with weights $p,q$ as the quantity $\area^{p,q}(K,L)=p \area(K \setminus L) + q \area(L \setminus K)$.
He proved that if for any $K \in \KK$, $\bar{a}_K^C(n,p,q)$ denotes the minimal weighted area deviation of $K$ and an arbitrary convex $n$-gon, then the sequence $\{ \bar{a}_K^C(n,p,q) \}$ is convex. Based on this idea, we introduce the following quantity.

Let $p,q \in [1,\infty]$ satisfy the equation $\frac{1}{p} + \frac{1}{q} = 1$, and let $C \in \KK_0$, $K,L \in \KK$.
We call the quantity
\[
\begin{aligned}
\perim_C^{p,q}(K,L) = p \left( \arclength_C(\bd(K) \setminus \inter(L))- \arclength_C(\bd(L) \cap K) \right) + \\
+ q  \left( \arclength_C(\bd(L) \setminus \inter(K)) - \arclength_C (\bd(K) \cap L) \right)
\end{aligned}
\]
the \emph{weighted $C$-perimeter deviation} of $K,L$ with weights $p,q$. Note that by convexity,
$\arclength_C(\bd(K) \setminus \inter(L)) \geq \arclength_C(\bd(L) \cap K)$ and $\arclength_C(\bd(L) \setminus \inter(K)) \geq \arclength_C (\bd(K) \cap L)$, with equality if and only if $K \subseteq L$ and $L \subseteq K$, respectively. Let $\bar{p}_K^C(n,p,q)$ denote the minimal $C$-perimeter deviation of $K$ and an arbitrary convex $n$-gon. We remark that if $K$ is $C$-convex, by replacing the convex $n$-gons in the definitions of $\bar{a}_K^C(n,p,q)$ and $\bar{p}_K^C(n,p,q)$ with $C$-$n$-gons, we may analogously define the quantities $\hat{a}_K^C(n,p,q)$ and $\hat{p}_K^C(n,p,q)$, respectively.
This leads to the following problems.

\begin{Problem}
Prove or disprove that for any $p,q \in [1,\infty ]$ with $\frac{1}{p} + \frac{1}{q} = 1$, $C \in \KK_o$ and $K \in \KK$, the sequence $\{ \bar{p}_K^C(n,p,q) \}$ is convex.
\end{Problem}                                                                                                            

\begin{Problem}
Prove or disprove that for any $p,q \in [1,\infty ]$ with $\frac{1}{p} + \frac{1}{q} = 1$, $C \in \KK_o$ and $C$-convex disk $K \in \KK$, the sequence $\{ \hat{p}_K^C(n,p,q) \}$ is convex. Does the same hold for $\{ \hat{a}_K^C(n,p,q) \}$ if $C$ is the Euclidean unit disk?
\end{Problem}

Before our last problem, we remark that $\hat{a}_K^C(n,1, \infty) = \area(K) - \hat{a}_K^C(n)$ and $\hat{a}_K^C(n,\infty,1) = \hat{A}_K^C(n)-\area(K)$.

\begin{Problem}
Is there a value $p_0 \in (1,\infty)$ such that for any $p$ with $p_0 < p \leq \infty$ and $q$ satisfying $\frac{1}{p} + \frac{1}{q} = 1$, for any  $C \in \KK_o$ and $C$-convex disk $K \in \KK$, the sequence $\{ \hat{a}_K^C(n,p,q) \}$ is convex?
\end{Problem}

Another direction of generalizing Dowker's theorems is to find the extremal values of geometric quantities of polygons circumscribed about /inscribed in a plane convex body which, in some sense, are \emph{self-defined} by the polygons (for a similar approach, see e.g. \cite{Langi}). We propose one problem in this direction.

\begin{Problem}\label{prob:volprod}
For any $K \in \KK_o$, let
\[
\hat{A}_m(K) = \inf \{ M(Q): Q \in \KK_o \hbox{ is a } (2m)-\hbox{gon circumscribed about } K \}, \hbox{ and}
\]
\[
\hat{a}_m(K) = \inf \{ M(Q): Q \in \KK_o \hbox{ is a } (2m)-\hbox{gon inscribed in } K \},
\]
where $M(Q)$ denotes the \emph{Mahler volume} or \emph{volume product} of $Q$, defined as the product of the areas of $Q$ and its polar. Find the convex disks $K$ for which the sequence $\{ \hat{A}_n(K) \}$ is convex, and for which $\{ \hat{a}_n(K) \}$ is concave.
\end{Problem}

Recall that volume product attains its maximum value $8$ over the whole family $\KK_o$ at $o$-symmetric parallelograms \cite{Mahler}. Thus, the suprema of the volume products of inscribed or circumscribed $o$-symmetric polygons are attained by inscribed/circumscribed parallelograms.

We note that the volume product of an $o$-symmetric convex body $C$ is closely related to the notion of \emph{Holmes-Thompson volume} of the normed space generated by $C$. For more information on it and its relation to volume product, the reader is referred to \cite{PaivaThompson}.

\section*{Acknowledgments}

The authors thank the referees for their helpful suggestions that significantly improved the quality of this paper, and one of the referees to bring the authors' attention to Dowker-type theorems for self-defined quantities.

\end{document}